\documentclass[11pt,reqno]{amsart}

\usepackage[leqno]{amsmath}
\usepackage{amsthm}
\usepackage{amsfonts, tikz-cd}
\usepackage{amssymb}
\usepackage{eucal}
\usepackage[all]{xy}
\usepackage{mathtools}
\usepackage{stmaryrd}
\usepackage{csquotes}
\usepackage{enumitem}
\usepackage{upgreek}
\usepackage{relsize}
\usepackage[bbgreekl]{mathbbol}
\usepackage{amsfonts}
\DeclareSymbolFontAlphabet{\mathbb}{AMSb} 
\DeclareSymbolFontAlphabet{\mathbbl}{bbold}
\newcommand{\Prism}{{\mathlarger{\mathbbl{\Delta}}}}

\setenumerate{itemsep=2pt,parsep=2pt,before={\parskip=2pt}}

\usepackage[colorlinks=true,hyperindex, linkcolor=magenta, pagebackref=false, citecolor=cyan,pdfpagelabels]{hyperref}

\usepackage{xspace}
\usepackage{epsfig,epic,eepic,latexsym,color}
\usepackage[all]{xy}
\usepackage{comment} 
\numberwithin{equation}{section}

\usepackage[english]{babel}
\usepackage{latexsym}

\newcommand{\nc}{\newcommand}
\nc{\rnc}{\renewcommand}
\rnc{\P}{\mathbf P}
\nc{\R}{\mathbf R}
\rnc{\rm}{\mathrm}
\rnc{\bf}{\mathbf}
\nc{\cal}{\mathcal}

\nc{\C}{\mathbf C}
\nc{\Q}{\mathbf Q}
\nc{\Z}{\mathbf Z}
\nc{\A}{\mathbf A}

\nc{\an}{\operatorname{an}}
\nc{\perfd}{\operatorname{perfd}}
\nc{\perf}{\operatorname{new, perf}}
\nc{\diam}{\diamondsuit}

\DeclareMathOperator{\charac}{char}

\nc{\htt}{\operatorname{ht}}
\nc{\Nm}{\operatorname{Nm}}
\nc{\Ker}{\operatorname{Ker}}
\nc{\mmod}{\operatorname{mod}}
\nc{\End}{\operatorname{End}}
\nc{\Tor}{\operatorname{Tor}}
\nc{\coker}{\operatorname{Coker}}
\nc{\dR}{\mathrm{dR}}
\nc{\crys}{\mathrm{crys}}
\nc{\dcrys}{\mathrm{crys}}
\nc{\cris}{\mathrm{cris}}
\nc{\Fil}{\mathrm{Fil}}
\nc{\gr}{\mathrm{gr}}
\nc{\conj}{\mathrm{conj}}

\nc{\Aut}{\operatorname{Aut}}
\nc{\cont}{\text{cont}}
\nc{\sep}{\text{sep}}
\nc{\Hom}{\mathrm{Hom}}
\nc{\Gal}{\mathrm{Gal}}
\nc{\Spec}{\text{Spec}\,}
\nc{\RZ}{\operatorname{RZ}}
\nc{\Syn}{\mathrm{Syn}}
\nc{\ProSyn}{\mathrm{ProSyn}}
\nc{\psyn}{\mathrm{psyn}}
\nc{\Psyn}{\mathrm{pSyn}}
\nc{\aff}{\mathrm{aff}}
\nc{\fppf}{\mathrm{fppf}}

\nc{\hocolim}{\operatorname{hocolim}}

\rnc{\t}{\tau}
\nc{\mm}{\pmb{\mu}}
\rnc{\a}{\alpha}
\nc{\n}{\mathfrak n}
\nc{\m}{\mathfrak m}
\nc{\mfs}{\mathfrak s}
\nc{\mf}{\mathfrak f}
\nc{\e}{\varepsilon}
\nc{\dd}{\delta}
\nc{\Imm}{\operatorname{Im}}
\rnc{\sp}{\operatorname{sp}}

\nc{\p}{\mathfrak p}
\nc{\q}{\mathfrak q}

\nc{\Sym}{\operatorname{Sym}}
\nc{\codim}{\operatorname{codim}}
\nc{\rk}{\operatorname{rk}}
\nc{\GL}{\operatorname{GL}}
\nc{\SL}{\operatorname{SL}}
\nc{\Lie}{\operatorname{Lie}}
\nc{\Ind}{\operatorname{Ind}}
\nc{\Div}{\underline{Div}}
\nc{\Pic}{\mathbf{Pic}}
\nc{\red}{\mathrm{red}}
\nc{\uPic}{\underline{ \mathbf{Pic}}}
\nc{\rH}{\mathrm{H}}
\nc{\Spf}{\operatorname{Spf}}
\nc{\Frac}{\operatorname{Frac}}
\nc{\colim}{\operatorname{colim}}
\nc{\Spa}{\operatorname{Spa}}
\rnc{\Spec}{\operatorname{Spec}}
\nc{\Alg}{\operatorname{Alg}}
\nc{\Poly}{\operatorname{Poly}}
\nc{\PShv}{\operatorname{PShv}}
\nc{\Shv}{\operatorname{Shv}}
\nc{\Fun}{\operatorname{Fun}}
\nc{\op}{\mathrm{op}}

\rnc{\an}{\operatorname{an}}
\nc{\et}{\text{\'et}}

\rnc{\et}{\text{\'et}}
\nc{\proet}{\text{pro\'et}}
\nc{\syn}{\text{syn}}
\nc{\prosyn}{\text{prosyn}}

\nc{\xr}{\xrightarrow}

\nc{\eps}{\epsilon}
\nc{\ov}{\overline}
\nc{\ud}{\underline}
\nc{\wdh}{\widehat}

\nc{\bG}{\mathbf G}
\nc{\bZ}{\mathbf Z}

\nc{\F}{\mathcal F}
\nc{\G}{\mathcal G}
\nc{\E}{\mathcal E}
\nc{\K}{\mathcal K}
\nc{\I}{\mathcal I}
\nc{\sQ}{\mathcal Q}

\nc{\X}{\mathcal X}
\nc{\Y}{\mathfrak Y}
\nc{\T}{\mathfrak T}

\nc{\LL}{\mathcal{L}}
\rnc{\S}{\mathcal S}
\nc{\M}{\mathcal M}
\nc{\sU}{\mathfrak U}
\nc{\V}{\mathfrak V}
\nc{\N}{\mathrm N}

\nc{\ra}{\rangle}

\nc{\os}{\overset}

\rnc{\O}{\mathcal O}
\nc{\J}{\mathcal J}
\theoremstyle{definition}

\newtheorem{thm}{Theorem}[subsection]
\newtheorem{lemma}[thm]{Lemma}
\newtheorem{defn}[thm]{Definition}

\newtheorem{rmk}[thm]{Remark}
\newtheorem{example}[thm]{Example}

\newtheorem{cor}[thm]{Corollary}

\newtheorem{question}[thm]{Question}

\begin{document}
\bibliographystyle{halpha-abbrv}
\title{Lefschetz theorems in flat cohomology and applications}
\author{Sean Cotner, Bogdan Zavyalov}
\begin{abstract} 
We prove a version of the Lefschetz hyperplane theorem for fppf cohomology with coefficients in any finite commutative group scheme over the ground field. As consequences, we establish new Lefschetz results for the Picard scheme.
\end{abstract}
\maketitle

\section{Introduction}

\subsection{Overview}

A Lefschetz hyperplane theorem asserts, roughly speaking, that the cohomology groups of a projective variety and one of its hyperplane sections agree in small degrees. There are many such results, e.g., for the Picard group \cite[Exp.\ XII, Cor.\ 3.6]{SGA2}, coherent cohomology \cite[III\textsubscript{1}, Thm.\ 1.3.1]{EGA}, and \'etale cohomology \cite[Exp.\ XIV, Cor.\ 3.3]{SGA4}. Our main result is a Lefschetz hyperplane theorem with coefficients in a finite commutative group scheme over a field.

\begin{thm}\label{thm:intro-main-1}(Theorem~\ref{thm:main-1}) Let $k$ be a field, let $Y$ be a projective syntomic $k$-scheme of pure dimension $N\geq d+1$, let $X\hookrightarrow Y$ be a closed syntomic subscheme, and let $G$ be a finite commutative $k$-group scheme. Then the cone 
\[
\rm{cone}\left(\rm{R}\Gamma_\rm{fppf}\left(Y, G\right) \to \rm{R}\Gamma_\rm{fppf}\left(X, G\right)\right)
\]
lies in $D^{\geq d}(\Z)$ if
\begin{enumerate}
    \item $Y = \bf{P}^N_k$ and $X$ is a global complete intersection of dimension $d$, or
    \item $X\subset Y$ is a strongly ample Cartier divisor (see Definition~\ref{defn:sufficiently-ample}). 
\end{enumerate} 
\end{thm}

The definition of strongly ample is somewhat technical, but ample Cartier divisors are automatically strongly ample in the following two situations (see Theorem~\ref{thm:derived-lefschetz}):
\begin{enumerate}
    \item $Y$ is smooth and the characteristic of $k$ is $0$;
    \item $Y$ is smooth and the characteristic of $k$ is $\geq d + 1$ and $Y$ lifts to $W_2(\ov{k})$.
\end{enumerate}
Moreover, Remark~\ref{rmk:strongly-ample-examples} shows that if $\mathcal{L}$ is an ample line bundle on $Y$ and either $Y$ has isolated singularities or $Y$ is a complete intersection inside a smooth projective variety, then for $n \gg 0$, any divisor defining $\mathcal{L}^n$ is strongly ample.

\begin{rmk} Example~\ref{exmpl:thm-fails-for-ample} shows that Theorem~\ref{thm:intro-main-1} may fail for ample (but not strongly ample) divisors $X\subset Y$. Example~\ref{example:fails-for-general-finite-group} shows that Theorem~\ref{thm:intro-main-1} may also fail for finite flat commutative $Y$-group schemes that are not defined over $k$. The assumption of syntomicity is similar to the assumptions on the Lefschetz theorems proved in \cite[Exp.\ XII, Cor.\ 3.6]{SGA2}.
\end{rmk}

We expect that there is a version of Theorem~\ref{thm:intro-main-1} for non-commutative finite $k$-group schemes, but we have not proved it. 

\begin{question}\label{question:non-commutative} Let $X\subset \bf{P}^N_k$ be a complete intersection of dimension at least $2$, and let $G$ be a finite (not necessarily commutative) $k$-group scheme. Is the natural morphism $\rm{H}^1(\bf{P}^N_k, G) \to \rm{H}^1(X, G)$ a bijection? The same question may be asked for $X \subset Y$ a strongly ample Cartier divisor in a projective syntomic
$k$-scheme $Y$ (possibly with a different definition of ``strongly ample'').
\end{question}

\begin{rmk} If both $X$, $Y$ are smooth connected projective $k$-schemes, and the ground field $k$ is algebraically closed, then Question~\ref{question:non-commutative} has a positive answer. This follows from the Lefschetz type result for Nori's fundamental group $\pi_1^{\rm{N}}(X, x)$ (see \cite[Th.~1.1]{Biswas-Holla}) and the observation that $\rm{H}^1(X, G) = \rm{Hom}_{k\text{-gp}}(\pi_1^{\rm{N}}(X, x), G)$ for any finite $k$-group scheme $G$ (see \cite[Prop.~3.11]{Nori}).
\end{rmk}

By devissage, Theorem~\ref{thm:intro-main-1} is reduced to the cases $G = \mu_\ell,\, \alpha_p,\, \mu_p,$ and $\bZ/p$, where $\ell$ is a prime number different from $p=\charac k$. The cases of $\alpha_p$ and $\bZ/p$ are reduced to questions of coherent cohomology using standard exact sequences. For $\ell \neq p$, the case of $\mu_\ell$ is settled using results in the theory of perverse sheaves. \smallskip

The case of $\mu_p$ will give us the most difficulty. Here we will find it convenient to pivot to proving a Lefschetz hyperplane theorem for the  syntomic cohomology of the Tate twists $\bZ_p(i)$. Using the Nygaard filtration, this will ultimately be reduced to proving a Lefschetz hyperplane theorem for each filtered piece in the conjugate filtration on de Rham cohomology, which has been established in \cite{ABM}. In particular, we get a Lefschetz hyperplane theorem for the syntomic cohomology of the Tate twists $\Z_p(i)$ defined in \cite{BMS2} (see also Section~\ref{section:terminology}): \smallskip

\begin{thm}\label{thm:intro-main-2}(Theorem~\ref{thm:hodge-implies-twists-lefschetz}) Let $k$ be a perfect field of characteristic $p>0$, let $Y$ be a projective syntomic $k$-scheme of pure dimension $N\geq d+1$, let $X\hookrightarrow Y$ be a closed syntomic subscheme, and let $i \geq 0$. Then the cone 
\[
C \coloneqq \rm{cone}\left(\rm{R}\Gamma_\syn\left(Y, \bf{Z}_p\left(i\right)\right) \to \rm{R}\Gamma_\syn\left(X, \bf{Z}_p\left(i\right)\right)\right)
\]
lies in $D^{\geq d}(\Z_p)$ with $\rm{H}^d(C)$ torsion-free if
\begin{enumerate}
    \item $Y = \bf{P}^N_k$ and $X$ is a global complete intersection of dimension $d$, or
    \item $X\subset Y$ is a strongly ample Cartier divisor (see Definition~\ref{defn:sufficiently-ample}). 
\end{enumerate}
\end{thm}

As a consequence of Theorem~\ref{thm:intro-main-1}, we prove a Lefschetz hyperplane theorem for the Picard schemes of projective syntomic $k$-schemes.

\begin{thm}\label{thm:intro-main-5}(Theorem~\ref{thm:SGA7-suff-ample})
Let $k$ be a field, let $Y$ be a projective syntomic $k$-scheme of pure dimension $d$, and let $X \subset Y$ be a strongly ample Cartier divisor. If $d \geq 3$, then $\Pic^\tau_{Y/k} \to \Pic^\tau_{X/k}$ is an isomorphism.
\end{thm}

In Corollary~\ref{cor:Pic-scheme-iso}, we combine Theorem~\ref{thm:intro-main-5} with \cite[Exp.\ XII, Cor.\ 3.6]{SGA2} to obtain a Lefschetz hyperplane theorem for the full Picard scheme. 

\begin{rmk} A.\,Langer has informed us that an effective version of Theorem~\ref{thm:intro-main-5} for $\Pic^\tau_{\rm{red}}$ follows from his Lefschetz type theorem for the $S$-fundamental group when $X$ and $Y$ are smooth (see \cite[Th.~10.2~and~10.4]{Langer}). 
\end{rmk}

Theorem~\ref{thm:intro-main-5} appears to be new in every dimension, even for smooth $X$ and $Y$. The main difficulty in deducing it from Theorem~\ref{thm:intro-main-1} is that the Picard schemes can be highly non-reduced in positive characteristic, so one cannot argue on the level of Picard groups, i.e., on the level of $k$-points. To overcome this issue, we need to use the structure theory of commutative group schemes over a field to obtain an isomorphism criterion (see Lemma~\ref{lemma:isomorphism-criterion}), and to verify the hypotheses of this criterion we need to generalize Theorem~\ref{thm:intro-main-1} to more general base schemes, at least for $G=\mu_p$ (see Corollary~\ref{cor:lefschetz-p-power-sufficiently-ample}). \smallskip

In the case of complete intersections in projective space, we can show that the hypothesis of strong ampleness is not necessary, and we can make a more refined statement, recovering \cite[Corollary 7.2.3]{Ces-Scholze}.

\begin{thm}\label{thm:intro-main-3}(\cite[Cor.~7.2.3]{Ces-Scholze}) Let $k$ be a field, and $X \subset \bf{P}^N_k$ be a complete intersection of dimension at least $2$. Then 
\begin{enumerate}
    \item\label{thm:SGA7-1} $\rm{Pic}(X)_{\rm{tors}}=0$;
    \item\label{thm:SGA7-2} the class of $\O_X(1)=i^*\O_{\bf{P}^N}(1)$ is a non-divisible element of $\rm{Pic}(X)$;
    \item\label{thm:SGA7-3} the group scheme $\bf{Pic}^\tau_{X/k}$ is trivial.
\end{enumerate}
\end{thm}

If $\dim X \geq 3$, then Theorem~\ref{thm:intro-main-3} was essentially settled by Grothendieck in \cite[Exp.\ XII, Cor.\ 3.6]{SGA2}. If $\dim X \geq 2$ and $X$ is smooth, then this was settled by Deligne in \cite[Exp.\ XI, Thm.\ 1.8]{SGA7_2}. A version for weighted complete intersection surfaces with certain limited singularities can be found in \cite[\textsection 1]{Lang-p-torsion}. The general case was established in \cite[Cor.~7.2.3]{Ces-Scholze}. However, Theorem~\ref{thm:intro-main-5} does not seem to follow from their methods. When we started writing this paper, we were not aware that Theorem~\ref{thm:intro-main-3} was proven in \cite{Ces-Scholze}. \smallskip

Both proofs of Theorem~\ref{thm:intro-main-3} share a similar idea of using prismatic techniques to reduce the study of flat cohomology of $\mu_p$ to studying the cohomology of certain coherent sheaves. However, the details of the proofs seem to be fairly different. Our proof is global and is based on the Lefschetz hyperplane theorem from \cite{ABM}, while the proof in \cite{Ces-Scholze} is local; in their argument they relate the Picard group of $X$ to the local Picard group of the vertex $x$ of the affine cone over $X$, and then use local techniques to study that Picard group. Namely, if $R$ is the local ring of $x$, $\rm{Pic}(X)/\Z[\O_X(1)]$ injects into $\rm{Pic}(\Spec R \smallsetminus \{x\})$, and this is good enough to prove Theorem~\ref{thm:intro-main-3}. However, the failure of the map $\rm{Pic}(X)/\Z[\O_X(1)] \to \rm{Pic}(\Spec R \smallsetminus \{x\})$ to be surjective is the main reason why their methods do not seem to be sufficient to obtain a proof of Theorem~\ref{thm:intro-main-5}. Both proofs treat all syntomic singularities uniformly in arbitrary dimension.

\subsection{Terminology}\label{section:terminology}  For a fixed prime $p$ and an object $M\in D(A)$, the {\it derived quotient} $[M/p]$ is the cone of the multiplication by $p$ map $M \xr{p} M$. \smallskip

For a field $k$ and a $k$-scheme $X$ (not necessarily of finite type), we denote the {\it derived de Rham cohomology} by $\rm{R}\Gamma_{\dR}(X/k)=\rm{R}\Gamma(X, \rm{dR}_{X/k})$; see \cite[\textsection VIII.2]{Illusie-cotangent} or \cite[\textsection 2]{bhatt-p-adic} for more details. We note that \cite[Corollary 3.10]{bhatt-p-adic} implies that $\rm{R}\Gamma_{\dR}(X/k) \simeq \rm{R}\Gamma(X, \Omega^\bullet_{X/k})$ if $X$ is smooth over a perfect field $k$ of characteristic $p>0$. However, these two complexes are usually different if $X$ is singular (or $k$ is of characteristic $0$). \smallskip

If $k$ is a perfect field of characteristic $p>0$ and $X$ is a $k$-scheme, we denote the (underived) {\it crystalline cohomology} by $\rm{R}\Gamma_\crys(X/W(k))$. We refer to \cite[\textsection 5]{BO} and \cite[\href{https://stacks.math.columbia.edu/tag/07GI}{Tag 07GI}]{stacks-project} for more details (see also \cite[Construction F.2]{Bhatt-Lurie}). \smallskip 

For an $\bf{F}_p$-scheme $X$ and an integer $i$, we define the {\it syntomic complex} $\rm{R}\Gamma_\syn(X, \Z_p(i))\in D(\Z_p)$ as in \cite[Variant 7.4.12]{Bhatt-Lurie}.

\subsection{Acknowledgements} 
We are very grateful to B.\,Bhatt for many fruitful discussions. In particular, we thank him for the suggestion to use the results from \cite{BMS2} to attack Theorem~\ref{thm:intro-main-1} and for explaining the relevant background. We are also grateful to D.\,Kubrak, S.\,Mondal, S.\,Petrov, and A.\,Prikhodko for many useful conversations. We thank A.\,Langer for several useful comments on an earlier draft. We thank the referee for many useful suggestions which greatly improved this paper. Finally, we would like to thank S.\,Naprienko for creating the website \href{https://thuses.com/}{Thuses.com}, without which this collaboration would likely not have taken place. BZ acknowledges funding through the Max Planck Institute for
Mathematics in Bonn, Germany, during the preparation of this work.

\section{Hodge $d$-equivalences and Kodaira pairs}

In this section, we recall the notion of Hodge $d$-equivalences and Kodaira pairs from \cite[\textsection 5]{ABM}. We further give some interesting examples that will be important for the rest of the paper. \smallskip

For the rest of the section, we fix a field $k$.

\subsection{Syntomic morphisms}\label{section:pro-syntomic}

The main goal of this section is to recall the definition of syntomic morphisms and discuss some of its basic properties that do not seem to be explicitly stated in the literature. All results of this section are certainly well-known to the experts. \smallskip

\begin{defn}{\cite[\href{https://stacks.math.columbia.edu/tag/00SL}{Tag 00SL}]{stacks-project}}\label{defn:syntomic-schemes} A morphism of schemes $f\colon X \to Y$ is {\it syntomic} if $f$ is flat, of finite presentation, and all fibers are local complete intersections in the sense of \cite[\href{https://stacks.math.columbia.edu/tag/00S9}{Tag 00S9}]{stacks-project}.
\end{defn}

\begin{example}\label{ex:finite-field-extension-syntomic} Any finite field extension $k\subset k'$ is syntomic by \cite[\href{https://stacks.math.columbia.edu/tag/00SF}{Tag 00SF}]{stacks-project}.
\end{example}

The next lemma provides the main source of examples of syntomic morphisms which we use.

\begin{lemma}\label{lemma:finite-flat-groups-syntomic} Let $A$ be a ring and let $G$ be a flat, finitely presented $A$-group scheme. Then $G$ is $A$-syntomic.
\end{lemma}
\begin{proof}
    This follows directly from \cite[\href{https://stacks.math.columbia.edu/tag/00SJ}{Tag 00SJ}]{stacks-project} and \cite[Exp.\ VII\textsubscript{B}, Cor.\ 5.5.1]{SGA3}.
\end{proof}

\begin{lemma}\label{lemma:syntomic-cotangent-complex} A morphism $f\colon A \to B$ is syntomic if and only if it is flat, finitely presented, and $L_{B/A}\in D(B)$ has Tor amplitude in $[-1, 0]$.
\end{lemma}

\begin{proof}
    If $f$ is syntomic, it is a locally complete intersection morphism by \cite[\href{https://stacks.math.columbia.edu/tag/069K}{Tag 069K}]{stacks-project}. Then $L_{B/A}$ has Tor amplitude in $[-1, 0]$ by \cite[\href{https://stacks.math.columbia.edu/tag/08SL}{Tag 08SL}]{stacks-project}. \smallskip

    Now we assume that $f$ is flat, finitely presented, and $L_{B/A}$ has Tor amplitude in $[-1, 0]$. We want to conclude that the fibers of $f$ are complete intersections. First, a standard limit argument using \cite[\href{https://stacks.math.columbia.edu/tag/08QQ}{Tag 08QQ}]{stacks-project} and \cite[Prop.~4.12]{Quillen} shows that $L_{B/A}\in D(B)$ is pseudo-coherent. In this case, \cite[\href{https://stacks.math.columbia.edu/tag/068V}{Tag 068V}]{stacks-project} and \cite[\href{https://stacks.math.columbia.edu/tag/08QQ}{Tag 08QQ}]{stacks-project} ensure that it suffices to prove the claim when $A=k$ is a field. Then the result follows from \cite[(1.2) Second Vanishing Theorem]{Avramov}.
\end{proof}

\begin{rmk} Lemma~\ref{lemma:syntomic-cotangent-complex} guarantees that Definition~\ref{defn:syntomic-schemes} coincides with the definition of syntomic morphisms given in \cite[Notation 2.1]{ABM}. In particular, all results of their paper are applicable with the definition of a syntomic morphism that we use.
\end{rmk}

\begin{cor}\label{cor:syntomic-quasi-syntomic} Let $k$ be a field of characteristic $p>0$, and let $R$ be a syntomic $k$-algebra. Then $R$ is quasisyntomic in the sense of \cite[Def.~C.6 and Ex.~C.11]{Bhatt-Lurie}.
\end{cor}
\begin{proof}
 First, \cite[Ex.~C.11]{Bhatt-Lurie} implies that it suffices to show that the map $L_{R/\bf{F}_p}$ is concentrated in degrees $[-1, 0]$. Using the fundamental exact triangle of cotangent complexes (see \cite[\href{https://stacks.math.columbia.edu/tag/08QX}{Tag 08QX}]{stacks-project}) and Lemma~\ref{lemma:syntomic-cotangent-complex}, it suffices to show that $L_{k/\bf{F}_p}$ is concentrated in degree $0$. Since any field extension $\bf{F}_p\subset k$ is a filtered colimit of smooth $\bf{F}_p$-algebras, the result follows from \cite[\href{https://stacks.math.columbia.edu/tag/08R5}{Tag 08R5}]{stacks-project} and \cite[\href{https://stacks.math.columbia.edu/tag/08S9}{Tag 08S9}]{stacks-project}.
\end{proof}

\subsection{Hodge $d$-equivalences}

We begin with a definition. 

\begin{defn}\label{defn:hodge-d-equivalence}\cite[Def.~5.1]{ABM} A morphism of syntomic $k$-schemes $f\colon X \to Y$ is a {\it Hodge $d$-equivalence} if, for every $s\geq 0$, we have that
\[
\rm{cone}\Big( \rm{R}\Gamma\left(Y, \wedge^s L_{Y/k}\right) \to \rm{R}\Gamma\left(X, \wedge^s L_{X/k}\right)\Big) \text{ lies in } D^{\geq d-s}(k).
\]
\end{defn}

Roughly, Definition~\ref{defn:hodge-d-equivalence} is a formal way to say that a morphism $f$ satisfies the conclusion of the Lefschetz hyperplane theorem for Hodge cohomology groups.\smallskip

The notion of Hodge $d$-equivalence will play a crucial role in our proof of Lefschetz-type results. In fact, our strategy for proving Theorem~\ref{thm:intro-main-1} for $G = \mu_p$ is to reduce to the analogous statement for Hodge cohomology, which follows from the results of \cite{ABM} on Hodge $d$-equivalences. \smallskip

To proceed, it is important to have a good supply of interesting Hodge $d$-equivalences, and we begin by providing some examples. In the next section, we will provide more examples after reviewing the related notion of Kodaira pairs from \cite{ABM}.

\begin{defn} We say that a finite type $k$-scheme $X$ {\it lifts to $W_2\big(\ov{k}\big)$} if there is a flat, finite type $W_2\big(\ov{k}\big)$-scheme $\widetilde{X}$ with an isomorphism of $\ov{k}$-schemes $\widetilde{X}_{\ov{k}} \simeq X_{\ov{k}}$.
\end{defn}

\begin{thm}\label{thm:derived-lefschetz} Let $i\colon X \hookrightarrow Y$ be a closed immersion of syntomic projective $k$-schemes.
\begin{enumerate}
    \item\label{thm:derived-lefschetz-1} If $Y = \bf{P}^N_k$ and $X \subset Y$ is a $d$-dimensional (global) complete intersection over $k$, then $i$ is a Hodge $d$-equivalence;
    \item If $k$ is a field of characteristic $0$, the $k$-scheme $Y$ is smooth projective of pure dimension $d+1$, and $X\subset Y$ is an ample Cartier divisor, then $i$ is a Hodge $d$-equivalence;
    \item If $k$ is a field of characteristic $p>0$, the $k$-scheme $Y$ is smooth projective of pure dimension $d+1$ and lifts to $W_2(\ov{k})$, and $X\subset Y$ is an ample Cartier divisor, then $i$ is a Hodge $\left(\rm{inf}(p, d+1)-1\right)$-equivalence.
\end{enumerate}
\end{thm}
\begin{proof}
    The first claim is \cite[Prop.~5.3]{ABM}, and the second follows from \cite[Ex.~5.6]{ABM} and \cite[Prop.~5.7]{ABM}. Using \cite[Cor.~2.8]{Deligne-Illusie}, the proof of the third claim is very similar to that of \cite[Prop.~5.7]{ABM}, but we spell out the details at the referee's request. \smallskip
    
    First, we can assume that $k=\ov{k}$ is algebraically closed. Let $n = \rm{inf}(p, d+1)$, and denote the ample line bundle $\O_Y(X)$ simply by $\O_Y(1)$. Then \cite[Cor.~2.8]{Deligne-Illusie} implies that
    \begin{equation}\label{eqn:Deligne-Illusie-vanishing}
        \rm{R}\Gamma(Y, \wedge^s L_{Y/k} \otimes \O_Y(-r)) \simeq \rm{R}\Gamma(Y, \Omega^s_{Y/k}(-r)) \in D^{\geq n-s}(k)
    \end{equation}
    for any $s\geq 0$ and all $r>0$. To finish the proof, it suffices to show the following claim: \smallskip
    
    {\it Claim.} For all $r, s\geq 0$, we have
    \begin{equation}\label{eqn:claim}
        C_{s, r} \coloneqq \rm{cone}\left( \rm{R}\Gamma(Y, \Omega^s_{Y/k}(-r)) \to \rm{R}\Gamma(X, \wedge^s L_{X/k}(-r))\right) \in D^{\geq n-s-1}(k).
    \end{equation}
    We prove Claim by induction on $s$.\smallskip
    
    {\it Base of induction. $s=0$.} In this case, we have a short exact sequence
    \begin{equation}\label{eqn:divisor-ses}
    0\to \O_Y(-r-1) \to \O_Y(-r) \to i_*\O_X(-r) \to 0. 
    \end{equation}
    Therefore, (\ref{eqn:Deligne-Illusie-vanishing}) implies that $C_{0, r} \simeq \rm{R}\Gamma\big(Y, \O_Y(-r-1)\big)[1] \in D^{\geq n-1}(k)$.\smallskip
    
    {\it Inductive step.} We fix an integer $s>0$ and assume that Claim is verified for $s-1$ and all $r> 0$. We first note that Claim for $s-1$ and Equation~(\ref{eqn:Deligne-Illusie-vanishing}) imply that, for $r>0$,
    \begin{equation}\label{eqn:induction-step}
    \rm{R}\Gamma(X, \wedge^{s-1} L_{X/k}(-r)) \in D^{\geq n-s}(k).
    \end{equation}

    For $r\geq 0$, we write the map $\rm{R}\Gamma(Y, \Omega^s_{Y/k}(-r)) \to \rm{R}\Gamma(X, \wedge^s L_{X/k}(-r))$ as the composition
    \[
    \rm{R}\Gamma(Y, \Omega^s_{Y/k}(-r)) \xr{f_s} \rm{R}\Gamma(X, i^*\Omega^s_{Y/k}(-r)) \xr{g_s} \rm{R}\Gamma(X, \wedge^s L_{X/k}(-r)).
    \]
    
    It suffices to see that the cones of $f_s$ and $g_s$ lie in $D^{\geq n-s-1}(k)$. Note (\ref{eqn:divisor-ses}) implies that the cone of $f_s$ can be identified with $\rm{R}\Gamma(Y, \Omega^s_{Y/k}(-r-1))[1]$. This complex lies in $D^{\geq n-s-1}(k)$ due to (\ref{eqn:Deligne-Illusie-vanishing}). This finishes the proof for $f_s$.\smallskip
    
    For the second map $g_s$, we note that the fundamental triangle of cotangent complexes (see \cite[(2.1.5.6) on p.\,138]{Illusie-cotangent}) and the computation $L_{X/Y} \simeq \O_X(-1)[1]$ (see \cite[Ch.\,III, Cor.\,3.2.7]{Illusie-cotangent}) imply that we have a distinguished triangle
    \[
    \O_X(-1) \to i^*\Omega^1_{Y/k} \to L_{X/k}.
    \]
    This allows us to regard $i^*\Omega^1_{Y/k}$ as a (two-term) filtered object in $D(X)$. Passing to wedge powers, we obtain a filtration on $i^*\Omega^s_{Y/k}$ (see \cite[Ch.\,V, Prop.\,4.2.5 and (4.2.7)]{Illusie-cotangent}). Since $\O_X(-1)$ is a line bundle, this filtration degenerates to a distinguished triangle
    \[
    \wedge^{s-1} L_{X/k}(-1) \to i^*\Omega^s_{Y/k} \to \wedge^s L_{X/k}.
    \]
    By twisting and taking derived global sections, we get the following distinguished triangle:
    \[
    \rm{R}\Gamma(X, \wedge^{s-1} L_{X/k}(-r-1)) \to \rm{R}\Gamma(X, i^*\Omega^s_{Y/k}(-r)) \to \rm{R}\Gamma(X, \wedge^s L_{X/k}(-r)). 
    \]
    Therefore, the claim follows from (\ref{eqn:induction-step}), which implies that
    \[
    \rm{cone}\left(g_s \right) \simeq \rm{R}\Gamma(X, \wedge^{s-1} L_{X/k}(-r-1))[1]\in D^{\geq n-s-1}(k). \qedhere 
    \]
\end{proof}

\subsection{Kodaira pairs}

In this subsection, we recall the definition of Kodaira pairs and use them discuss some other examples of Hodge $d$-equivalences which will be important for us.

\begin{defn}\label{defn:kodaira-pair}\cite[Def.~5.5]{ABM}  A Kodaira pair is a $d$-dimensional $k$-scheme $Y$ and an ample line bundle $\cal{L}$ such that $\rm{R}\Gamma(Y, \wedge^s L_{Y/k} \otimes \cal{L}^{-r})\in D^{\geq d-s}(k)$ for all $s\geq 0$ and all $r>0$.
\end{defn}

The main point of Definition~\ref{defn:kodaira-pair} is that, if $(Y, \cal{L})$ is a Kodaira pair, then \cite[Prop.~5.7]{ABM} ensures that any section $s\in \cal{L}(Y)$ defines a Hodge $(d-1)$-equivalence $H=\rm{V}_Y(s) \hookrightarrow Y$.\smallskip

The main goal of this section is to construct enough ``asymptotic" examples of Kodaira pairs. For this, we will need a version of Fujita's Vanishing Theorem for nef bundles. We recall that a vector bundle $\cal{E}$ on a projective $k$-scheme $X$ is {\it nef} if the tautological bundle $\O_{\bf{P}_X(\cal{E})/X}(1)$ is a nef line bundle on $\bf{P}_X(\cal{E})$ (see \cite[Def.~1.4.1]{positivity-1} and \cite[Def.~6.1.1]{positivity-2}).

\begin{lemma}\label{lemma:vanishing-ample} Let $X$ be a projective $k$-scheme, $\F$ a finite rank vector bundle on $X$, $\cal{E}$ a nef vector bundle on $X$, and $\cal{L}$ an ample line bundle on $X$. Then there is an integer $n_0\geq 0$ such that 
\[ 
\rm{H}^i\left(X, \F\otimes \Sym^a(\cal{E}) \otimes \cal{L}^{\otimes b}\right)=0
\]
for $i>0$, $a, b\geq n_0$. 
\end{lemma}
\begin{proof}
    We consider the morphism $p\colon P\coloneqq \bf{P}_X(\cal{E}) \to X$ with antitautological bundle $\O_P(1)$. Since $\cal{E}$ is nef, the line bundle $\O_P(1)$ is also nef, so \cite[Ex.~1.4.4(i)]{positivity-1} and the fact that tensor products of nef line bundles are nef ensure that $\O_P(n) \otimes p^*\cal{L}^{\otimes n'}$ is nef for any integers $n,n'\geq 0$. Furthermore, \cite[\href{https://stacks.math.columbia.edu/tag/0892}{Tag 0892}]{stacks-project} guarantees that there is an integer $k>0$ such that $\O_P(1)\otimes p^*\cal{L}^{\otimes \ell}$ is ample for all $\ell \geq k$. Therefore, Fujita's vanishing (see \cite[\textsection 5, Theorem]{Fujita} or \cite[Th.~1.5]{Keeler} and \cite{Keeler-erratum}) implies that there is a constant $m_0\geq 0$ such that
    $\rm{H}^i\left(P, p^*\big(\F\otimes \cal{L}^{\otimes b}\big) \otimes \O_P(a)\right)=0$
    for $a\geq m_0$, $b\geq km_0$, and $i>0$. Now the projection formula and Serre's computation (see \cite[\href{https://stacks.math.columbia.edu/tag/01E8}{Tag 01E8}]{stacks-project} and \cite[\href{https://stacks.math.columbia.edu/tag/01XX}{Tag 01XX}]{stacks-project}) imply that 
    \[
    \rm{R}p_*\left(p^*\big(\F \otimes \cal{L}^{\otimes b}\big)\otimes \O_P(a)\right) \simeq \F\otimes \cal{L}^{\otimes b} \otimes^L \rm{R}p_*\O_P(a) \simeq \F\otimes \cal{L}^{\otimes b} \otimes \Sym^a\left(\cal{E} \right). 
    \]
    Therefore, we conclude that
    \[
    \rm{R}\Gamma\left(X, \F \otimes \Sym^a\left(\cal{E}\right) \otimes \cal{L}^{\otimes b}\right) \simeq \rm{R}\Gamma\left(P, p^*(\F\otimes \cal{L}^{\otimes b}) \otimes \O_P(a)\right) \in D^{\leq 0}(k)
    \]
    for $a\geq m_0$, $b\geq km_0$. This gives the desired vanishing by setting $n_0=km_0$.
\end{proof}

The following lemma gives an important source of examples of Kodaira pairs.

\begin{lemma}\label{lemma:nef-normal-bundle} Let $Y$ be a smooth projective $k$-scheme of dimension $N$, let $X$ be a syntomic $k$-scheme of pure dimension $d$, and let $\iota \colon X \to Y$ be a closed immersion of $k$-schemes. Let $\cal{I}_X$ be the ideal sheaf of $X$ in $Y$, and $\cal{L}$ an ample line bundle on $X$. If the normal bundle $\cal{N}\coloneqq \cal{N}_{X/Y}=(\cal{I}_X/\cal{I}_X^2)^{\vee}$ is nef, then there is an integer $r$ such that $(X, \cal{L}^n)$ is a Kodaira pair for any $n\geq r$. 
\end{lemma}
    In the proof below, we always use the notation $\Sym^n, \wedge^n, \Gamma^n$, and $(-)^{\vee}=\rm{R}\ud{\cal{H}om}_X(-, \O_X)$ in the derived sense (see \cite[Ch.\,V, \textsection 4]{Illusie-cotangent} or  \cite[\textsection 25.2]{Lurie-spectral} for the construction and basic properties of these functors). These functors coincide with their naive analogues when applied to vector bundles. We recall that by definition, if $\mathcal{E}$ is a vector bundle in degree $0$, then $\Gamma^n(\mathcal{E}) \coloneqq (\mathcal{E}^{\otimes n})^{\Sigma_n}$.
\begin{proof}
    First, we note that \cite[\href{https://stacks.math.columbia.edu/tag/0DWA}{Tag 0DWA}]{stacks-project} implies that $X$ is a Gorenstein scheme. Therefore, \cite[\href{https://stacks.math.columbia.edu/tag/0FVV}{Tag 0FVV}]{stacks-project} and \cite[\href{https://stacks.math.columbia.edu/tag/0BFQ}{Tag 0BFQ}]{stacks-project} imply that there is a {\it line bundle} $\omega_X$ such that the dualizing complex $\omega^\bullet_X$ is isomorphic to $\omega_X[d]$. Thus, Grothendieck duality (see \cite[\href{https://stacks.math.columbia.edu/tag/0B6I}{Tag 0B6I}]{stacks-project}) and perfectness of $\wedge^s L_{X/k}$ imply that
    \[
    \rm{R}\Gamma(X, \wedge^s L_{X/k} \otimes \cal{L}^{-n}) = \rm{R}\Gamma(X, \omega_X[d] \otimes (\wedge^s L_{X/k})^\vee \otimes \cal{L}^n)^{\vee}
    \]
    for all $s$ and $n$. Therefore, it suffices to show that there is an integer $r$ such that 
    \begin{equation}\label{equation:reduction-step-1}
    \rm{R}\Gamma(X, \omega_X \otimes (\wedge^s L_{X/k})^\vee \otimes \cal{L}^n) \in D^{\leq s}(k)
    \end{equation}
    for any $s\geq 0$ and $n\geq r$. \smallskip

    Note that \cite[Ch.\,III, Cor.~3.2.7]{Illusie-cotangent} implies that there is an isomorphism
    \[
    L_{X/k} \simeq [0 \to \cal{I}_X/\cal{I}_X^2 \to \Omega^1_{Y/k}|_X \to 0],
    \]
    where $\Omega^1_{Y/k}|_X$ is in degree $0$. Therefore, this allows us to regard $L_{X/k}$ as a (two-term) filtered object in $D(X)$ with associated graded pieces $\rm{gr}^0 L_{X/k}= \cal{I}_X/\cal{I}_X^2[1] =\cal{N}^{\vee}[1]$ and $\rm{gr}^1 L_{X/k}= \Omega^1_{Y/k}|_{X}$. \smallskip
    
    Now \cite[Ch.\ V, Prop.\ 4.2.5 and (4.2.7)]{Illusie-cotangent} implies that $\wedge^s L_{X/k}$ admits a finite decreasing filtration with associated graded pieces gives by $\rm{gr}^i\wedge^s L_{X/k} \simeq \wedge^{s-i}\left(\cal{N}^{\vee}[1]\right) \otimes^L \wedge^i\left(\Omega^1_{Y/k}|_X\right)$ for $i=0, \dots, s$. For brevity, we denote $\wedge^i(\Omega^1_{Y/k}|_X)$ by $\cal{G}_i$. \smallskip

    Now we use \cite[Rmk.\ 25.2.2.4, Prop.\ 25.2.4.2]{Lurie-spectral} to write $\wedge^j\left(\cal{N}^{\vee}[1]\right) \simeq \Gamma^j\left(\cal{N}^{\vee}\right)[j] \simeq \Sym^j(\cal{N})^{\vee}[j]$. Passing to duals, we get a finite (increasing) filtration on $\left(\wedge^s L_{X/k}\right)^\vee$ with associated graded pieces
    \begin{align}\label{equation:reduction-step-2}
    \rm{gr}_i \left(\left(\wedge^s L_{X/k}\right)^\vee\right) \simeq \left(\rm{gr}^{s-i} \wedge^s L_{X/k}\right)^{\vee} \simeq \left(\Gamma^i(\cal{N}^{\vee})[i]\right)^\vee\otimes^L \left(\cal{G}_{s-i}\right)^\vee \simeq \Sym^{i}\left(\cal{N}\right) \otimes \cal{G}_{s-i}^\vee[-i]
    \end{align}

    for $i=0, \dots, s$. Combining (\ref{equation:reduction-step-1}) and (\ref{equation:reduction-step-2}), we conclude that it suffices to find an integer $r\geq 0$ such that
    \begin{equation}\label{eqn:desired-vanishing}
    \rm{R}\Gamma(X, \omega_X \otimes \Sym^{i}\left(\cal{N}\right) \otimes \cal{G}_{s-i}^\vee \otimes \cal{L}^n) \in D^{\leq 0}(k)
    \end{equation}
    for any $s\geq 0$, $0 \leq i \leq s$, and $n\geq r$. For this, we apply Lemma~\ref{lemma:vanishing-ample} to the vector bundles $\F_j\coloneqq \omega_X \otimes \cal{G}_{j}^\vee$, the nef vector bundle $\cal{N}$, and the ample line bundle $\cal{L}$ to find $r_0, \dots, r_N$ such that
    \begin{equation}\label{eqn:vanishing-asymptotical}
    \rm{R}\Gamma(X, \omega_X\otimes \Sym^{a}\left(\cal{N}\right) \otimes \cal{G}_{j}^\vee \otimes \cal{L}^b)\in D^{\leq 0}(k)
    \end{equation}
    for any $0 \leq j \leq N$ and $a,b \geq r_i$. We set $r'\coloneqq \rm{max}(r_0, \dots, r_N)$. Serre vanishing \cite[\href{https://stacks.math.columbia.edu/tag/0B5U}{Tag 0B5U}]{stacks-project} ensures that, for $0 \leq j \leq N$ and $0 \leq t \leq r'-1$, we can find an integer $r_{i, j}$ such that 
    \begin{equation}\label{eqn:vanishing-for-small-i-j}
    \rm{R}\Gamma(X, \omega_X\otimes \Sym^{t}\left(\cal{N}\right) \otimes \cal{G}_{j}^\vee \otimes \cal{L}^n)\in D^{\leq 0}(k)
    \end{equation}
    for $n\geq r_{i, j}$. Using (\ref{eqn:vanishing-asymptotical}), (\ref{eqn:vanishing-for-small-i-j}), and the fact that $\cal{G}_{s-i}^\vee=0$ for $s-i>N$, we see that $r=\rm{max}(r', r_{i, j})$ does the job.
\end{proof}

Now we are ready to give more examples of Hodge $d$-equivalences:

\begin{cor}\label{cor:example-nef-normal-bundle} Let $\iota\colon X \to Y$ be a closed immersion as in Lemma~\ref{lemma:nef-normal-bundle}, and let $\cal{L}$ be an ample line bundle on $X$. Then there is an integer $n_0$ such that for all $n\geq n_0$ and any effective Cartier divisor $X\subset Y$ defined by a section of $\cal{L}^n$, the morphism $X \to Y$ is a Hodge $d$-equivalence in either of the following two cases:
\begin{enumerate}
    \item $X$ has isolated singularities and $Y=\bf{P}^N_k$;
    \item $Y$ has pure dimension $N$, and there is a nef rank $N-d$ vector bundle $\cal{E}$ on $Y$ with a section $s\in \Gamma(Y, \cal{E})$ such that $X=\rm{V}_Y(s)$.
\end{enumerate}
\end{cor}
\begin{proof}
    Lemma~\ref{lemma:nef-normal-bundle} and \cite[Prop.~5.7]{ABM} ensure that it suffices to show that the normal bundle $\cal{N}_{X/Y}=\left(\cal{I}_X/\cal{I}_X^2\right)^{\vee}$ is nef in either of these two cases. In the first case, this follows directly from \cite[Th.~2.16]{thickenings} and the observation that ample vector bundles are nef. \smallskip

    In the second case, we consider the Koszul complex $\rm{Kos}(\cal{E}^\vee; s)$ associated to the morphism $s\colon \cal{E}^\vee \to \O_Y$ \cite[\href{https://stacks.math.columbia.edu/tag/062K}{Tag 062K}]{stacks-project}, \cite[App.\ A.5]{Fulton}. By construction, $\cal{H}^0\left(\rm{Kos}(\cal{E}^\vee; s)\right) \simeq \O_Y/\cal{I}_X$. Furthermore, \cite[\href{https://stacks.math.columbia.edu/tag/02JN}{Tag 02JN}]{stacks-project} and \cite[\href{https://stacks.math.columbia.edu/tag/063I}{Tag 063I}]{stacks-project} imply that $\cal{H}^i\left(\rm{Kos}(\cal{E}^\vee; s) \right)=0$ for $i > 0$, so 
    \[
    \rm{Kos}^{\leq -1}(\cal{E}^\vee; s) \to \cal{I}_X
    \]
    is a finite resolution of $\cal{I}_X$ by vector bundles. Since all differentials in $\rm{Kos}(\cal{E}^\vee; s)$ vanish on $X$, we conclude that $\cal{I}_X/\cal{I}_X^2 \simeq \cal{E}^\vee|_X$. Therefore $\cal{N}_{X/Y}=\cal{E}|_X$ is nef by \cite[Prop.~6.1.2(ii)]{positivity-2}.
\end{proof}

Corollary~\ref{cor:example-nef-normal-bundle} motivates the following definition:

\begin{defn}\label{defn:sufficiently-ample} An effective Cartier divisor $D\subset X$ in a projective $k$-scheme of pure dimension $d+1$ is {\it strongly ample} if $D \to X$ is a Hodge $d$-equivalence and $X \smallsetminus D$ is an an affine subscheme. 
\end{defn}

\begin{rmk}\label{rmk:strongly-ample-examples} Corollary~\ref{cor:example-nef-normal-bundle} implies that an effective divisor $D\subset X$ defined by a sufficiently high power of an ample line bundle is strongly ample provided that $X$ has isolated singularities, or can be realized as a ``complete intersection'' inside a smooth projective variety. We do not know whether an analogous statement holds for an arbitrary projective syntomic $k$-scheme $X$. 
\end{rmk}

\begin{rmk} Theorem~\ref{thm:derived-lefschetz} implies that any ample Cartier divisor in a smooth projective variety $Y$ is strongly ample in either of the following situations:
\begin{enumerate}
    \item the ground field $k$ is of characteristic $0$;
    \item the ground field $k$ is of characteristic $p>0$, $\dim Y\leq p$, and $Y$ admits a lift over $W_2\left(\ov{k}\right)$.
\end{enumerate}
\end{rmk}

\section{Lefschetz hyperplane theorem for flat cohomology}

\subsection{$\bf{Z}_p(i)$ and $\mu_p$ coefficients}\label{section:lefschetz-p-part}

In this section, we prove the Lefschetz hyperplane theorem for $\mu_p$-cohomology groups for Hodge $d$-equivalences over a perfect field. More generally, we show it for $\bf{Z}_p(i)$-cohomology groups for all $i\geq 0$. \smallskip

For the rest of the section, we fix a perfect field $k$ of characteristic $p>0$.

\begin{thm}\label{thm:hodge-implies-twists-lefschetz} Let $X \to Y$ be a Hodge $d$-equivalence of syntomic $k$-schemes. For $i\geq 0$, the cone 
\[
C\coloneqq \rm{cone}\left(\rm{R}\Gamma_\syn\left(Y, \bf{Z}_p\left(i\right)\right) \to \rm{R}\Gamma_\syn\left(X, \bf{Z}_p\left(i\right)\right)\right)
\]
lies in $D^{\geq d}(\bf{Z}_p)$ and $\rm{H}^d(C)$ is torsion-free. 
\end{thm}

We will give a proof shortly, but before doing so we discuss its main application for our purposes. 

\begin{cor}\label{cor:lefschetz-flat-mu_p} Let $X \to Y$ be a Hodge $d$-equivalence of syntomic $k$-schemes. Then 
\[
C \coloneqq \rm{cone}\left(\rm{R}\Gamma_\fppf\left(Y, \mu_p\right) \to \rm{R}\Gamma_\fppf\left(X, \mu_p\right)\right) \in D^{\geq d}(\bf{F}_p).
\]
\end{cor}
\begin{proof}
    We note that \cite[Prop.~7.5.6]{Bhatt-Lurie} and \cite[Th.~11.7]{Brauer_3} imply that 
    \[
    \big[\rm{R}\Gamma_\syn\left(Y, \bf{Z}_p\left(1\right)\right)/p\big]  \simeq \big[\rm{R}\Gamma_{\rm{fppf}}(Y, \bf{G}_m)/p\big][-1] \simeq \rm{R}\Gamma_{\rm{fppf}}(Y, \mu_p)
    \]
    and similarly for $X$. Combining these observations with Theorem~\ref{thm:hodge-implies-twists-lefschetz}, we conclude. 
\end{proof}

Now we turn to the proof of Theorem~\ref{thm:hodge-implies-twists-lefschetz}. The main idea of the proof is to deduce it through a series of reductions from the definition of Hodge $d$-equivalences. \smallskip

For a syntomic $k$-scheme $X$, we refer to Corollary~\ref{cor:syntomic-quasi-syntomic}, \cite[Warning 4.6.2, Prop.~5.1.1]{Bhatt-Lurie}\footnote{Strictly speaking, \cite[Prop.~5.1.1]{Bhatt-Lurie} constructs the Nygaard filtration only in the affine case. However, similarly to \cite[Not.~5.5.23]{Bhatt-Lurie}, this can be formally extended to an arbitrary $k$-scheme.} for the definition of the {\it Nygaard filtration} 
\[
\rm{Fil}_\N^\bullet \rm{R}\Gamma_\crys(X/W(k)) \simeq \rm{Fil}_\N^\bullet F^*\rm{R}\Gamma_{\Prism}(X/W(k)).
\]

\begin{lemma}\label{lemma:hodge-equivalence-implies-nygaard} Let $X \subset Y$ be a Hodge $d$-equivalence of syntomic $k$-schemes. For $i\geq 0$, the cone
\[
C\coloneqq \rm{cone}\left(\Fil_\N^i\rm{R}\Gamma_\dcrys\left(Y/W(k)\right) \to \Fil_\N^i \rm{R}\Gamma_\dcrys\left(X/W(k)\right)\right)
\]
lies in $D^{\geq d}(\bf{Z}_p)$ and $\rm{H}^d(C)$ is torsion-free. 
\end{lemma}
\begin{proof}
We argue by induction on $i\geq 0$. The case of $i=0$ is clear from \cite[Rem.~5.2]{ABM} since $\Fil_\N^0 \rm{R}\Gamma_\dcrys(X/W(k))\simeq \rm{R}\Gamma_\dcrys(X/W(k))$ and the same applies to $Y$. \smallskip

Now fix $i \geq 0$ and suppose we know the claim for $i$. The global version of \cite[Rem.~5.1.2]{Bhatt-Lurie} and the de Rham comparison Theorem (see \cite[Prop.~5.2.5]{Bhatt-Lurie}) imply\footnote{In the formula below, we implicitly use that the Frobenius morphism $F\colon W(k) \to W(k)$ is an isomorphism and preserves the ideal $(p)\subset W(k)$. Therefore, the formula in \cite[Prop.~5.2.5]{Bhatt-Lurie} can be Frobenius ``untwisted''. We also use that the Breuil-Kisin twists can be canonically trivialized for the prism $(W(k), (p))$.} that
\[
\rm{gr}^i_\N\rm{R}\Gamma_{\dcrys}(X/W(k)) \simeq \rm{gr}^i_\N F^*\rm{R}\Gamma_{\Prism}(X/W(k)) \simeq \rm{Fil}_i^{\rm{\conj}} \rm{R}\Gamma_{\overline{\Prism}}(X/W(k)) \simeq \rm{Fil}_i^{\rm{conj}} (F^*)^{-1} \rm{R}\Gamma(X, \rm{dR}_{X/k}),
\]
where $\rm{dR}_{X/k}$ is the derived de Rham complex of $X$ and $F\colon W(k) \to W(k)$ is the Frobenius morphism. Since the Frobenius morphism $F$ is an isomorphism and $\Z_p$-linear, we can choose a $\Z_p$-linear isomorphism $(F^*)^{-1} \rm{R}\Gamma(X, \rm{dR}_{X/k}) \simeq \rm{R}\Gamma(X, \rm{dR}_{X/k})$. Therefore, we have the following commutative diagram of exact triangles in $D(\Z_p)$:
\[
\begin{tikzcd}
\Fil_\N^{i+1}\rm{R}\Gamma_\dcrys\left(Y/W(k)\right)\arrow{d} \arrow{r} & \Fil_\N^i\rm{R}\Gamma_\dcrys\left(Y/W(k)\right) \arrow{r}{} \arrow{d} & \Fil^\conj_i\rm{R}\Gamma_{\dR}(Y/k) \arrow{d} \\
\Fil_\N^{i+1}\rm{R}\Gamma_\dcrys\left(X/W(k)\right) \arrow{r} & \Fil_\N^i\rm{R}\Gamma_\dcrys\left(X/W(k)\right) \arrow{r}{} & \Fil^\conj_i\rm{R}\Gamma_{\dR}(X/k),
\end{tikzcd}
\]
where $\rm{Fil}^\bullet_{\rm{conj}}\rm{R}\Gamma_{\dR}(X/k)$ is the conjugate filtration on the derived de Rham complex of $X$ (and similarly for $Y$). Denote by $C$, $C'$, and $C''$ cones of the left, middle, and right vertical maps respectively. Now \cite[Rem.~5.2]{ABM} shows that $C'' \in D^{\geq d}(\Z_p)$ and the induction hypothesis gives that $C'\in D^{\geq d}(\Z_p)$ with $\rm{H}^d(C')$ torsion-free. This formally implies that $C\in D^{\geq d}(\Z_p)$ and that $\rm{H}^d(C)$ is torsion-free.  
\end{proof}

Now we are ready to prove Theorem~\ref{thm:hodge-implies-twists-lefschetz}.

\begin{proof}[Proof of Theorem~\ref{thm:hodge-implies-twists-lefschetz}]
    We note that \cite[Th.~5.6.2 and Var.~7.4.12]{Bhatt-Lurie} imply that we have the following commutative diagram of exact triangles: 
\[
    \begin{tikzcd}
    \rm{R}\Gamma_\syn\left(Y, \bf{Z}_p(i) \right) \arrow{d} \arrow{r} & \Fil_\N^i\rm{R}\Gamma_\dcrys\left(Y/W(k)\right) \arrow{r}{} \arrow{d} & \rm{R}\Gamma_\dcrys(Y/W(k)) \arrow{d} \\
    \rm{R}\Gamma_\syn\left(X, \bf{Z}_p(i) \right) \arrow{r} & \Fil_\N^i\rm{R}\Gamma_\dcrys\left(X/W(k)\right) \arrow{r}{} & \rm{R}\Gamma_\dcrys(X/W(k)).
    \end{tikzcd}
\]
Now, as in the proof of Lemma~\ref{lemma:hodge-equivalence-implies-nygaard}, we see that it suffices to prove the claim for each $\Fil_{\N}^i\rm{R}\Gamma_\crys(-/W(k))$ (including $i=0$). This follows directly from Lemma~\ref{lemma:hodge-equivalence-implies-nygaard}.
\end{proof}

\subsection{$\mu_\ell$ coefficients}\label{section:lefschetz-non-p-part}

In this section, we give a proof of the Lefschetz hyperplane theorem for $\mu_\ell$ coefficients. The proof is probably well-known to the experts, but it seems hard to extract from the literature. The main difficulty is that we do not require the ambient space $Y$ to be smooth, but only syntomic (also, see \cite[Appendix B]{Poonen-Voloch} for the case of a smooth ambient space $Y$).  \smallskip

For the rest of the section, we fix a separably closed field $k$ (possibly of characteristic $0$) and a prime number $\ell$ not equal to the characteristic of $k$. \smallskip

We recall that there is a well-behaved theory of perverse $\bf{F}_\ell$-sheaves on finite type $k$-schemes; see \cite[Intro to Ch.\ 4]{BBD} or \cite[\textsection 4]{Bhatt-Hansen}\footnote{This is written for rigid-analytic spaces, but similar (and, in fact, easier) proofs work in the algebraic situation.} for a more detailed discussion. We only mention two main results that we will need in this section.

\begin{lemma}\label{lemma:perverse-properties} Let $X$ a finite type $k$-scheme of pure dimension $d$. Then
\begin{enumerate}
    \item the sheaf $\ud{\bf{F}}_\ell[d]$ is a perverse sheaf on $X$ if $X$ is $k$-syntomic;
    \item for a perverse $\bf{F}_\ell$-sheaf $\cal{L}$, the complex $\rm{R}\Gamma_c(X_\et, \cal{L})$ lies in $D^{\geq 0}(\bf{F}_\ell)$ if $X$ is affine. 
\end{enumerate}
\end{lemma}
\begin{proof}
    The first claim is \cite[Cor.\ 1.4]{Il-perverse}; the second is \cite[Thm.\ 2.4]{Il-perverse} or \cite[Thm.\ 4.1.1]{BBD}.
\end{proof}

For our next result, we drop the assumption that $k$ is separably closed:

\begin{thm}\label{thm:etale-lefschetz} Let $k$ be a field, let $Y$ be a syntomic proper $k$-scheme of pure dimension $d+1$, let $X\subset Y$ be a Cartier divisor such that $Y\smallsetminus X$ is affine (e.g., $X$ is an ample Cartier divisor), let $\ell$ be a prime different from $\charac k$, and let $G$ be a finite $\ell^\infty$-torsion $k$-group scheme. Then
\[
C\coloneqq \rm{cone}\left(\rm{R}\Gamma_{\fppf}(Y, G) \to \rm{R}\Gamma_{\fppf}(X, G) \right) \in D^{\geq d}(\Z).
\] 
\end{thm}
\begin{proof}
    In this proof, we feel freely use that $G$ is \'etale, and so fppf cohomology with $G$-coefficients coincide with analogous \'etale cohomology (see \cite[Th.~11.7]{Brauer_3}). \smallskip

    We first assume that $k$ is separably closed. In this case, $G$ has a finite filtration with associated graded pieces isomorphism to $\ud{\bf{F}}_\ell$. This it suffices to prove the claim for $G=\ud{\bf{F}}_\ell$. In this case, we then denote the complement of $X$ in $Y$ by $U$. Then \cite[\href{https://stacks.math.columbia.edu/tag/0GKP}{Tag 0GKP}]{stacks-project} and properness of $Y$ imply that we have an exact triangle
    \[
    \rm{R}\Gamma_c(U_\et, \ud{\bf{F}}_\ell) \to \rm{R}\Gamma_\et(Y, \ud{\bf{F}}_\ell) \to \rm{R}\Gamma_\et(X, \ud{\bf{F}}_\ell).
    \]
    By Lemma~\ref{lemma:perverse-properties}, $\ud{\bf{F}}_\ell[d+1]$ is a perverse sheaf on $U$. Therefore the same lemma implies that $\rm{R}\Gamma_c(U_\et, \ud{\bf{F}}_\ell) \in D^{\geq d+1}(\Z)$, so $C$ lies in $D^{\geq d}(\Z)$. \smallskip

    If $k$ is an arbitrary field, we denote its absolute Galois group by $\rm{Gal}_k$ and note that 
    \[
    \rm{R}\Gamma_\et(X, G) \simeq \rm{R}\Gamma_{\rm{cont}}(\rm{Gal}_k, \rm{R}\Gamma_\et(X_{k^{\rm{sep}}}, G_{k^{\rm{sep}}})),
    \]
    and similarly for $Y$. Thus the general result follows from the special case $k = k^{\rm{sep}}$. 
\end{proof}

\subsection{Finite flat commutative group scheme coefficients}

In this section, we prove the general version of the Lefschetz hyperplane theorem. The strategy is to reduce the general case to the cases of finite flat group schemes $G=\mu_\ell$, $\mu_p$, $\alpha_p$, and $\bZ/p$, and deal with each case separately.

\begin{lemma}\label{lemma:lefschetz-p-power-perfect-field} Let $k$ be a perfect field of characteristic $p> 0$, let $X \to Y$ be a Hodge $d$-equivalence of syntomic $k$-schemes, and let $G$ be a commutative finite flat $k$-group scheme with a finite filtration $\rm{Fil}^\bullet G$ such that each $\rm{gr}^i G$ is isomorphic to either $\mu_p$, $\alpha_p$, or $\bf{Z}/p$. Then 
\[
\rm{cone}\left(\rm{R}\Gamma_{\rm{fppf}}\left(Y, G\right) \to \rm{R}\Gamma_{\rm{fppf}}\left(X, G\right)\right) \in D^{\geq d}(\bf{Z}).
\]
\end{lemma}
\begin{proof}
    One easily reduces to the case $G=\mu_p$, $G=\alpha_p$, or $G=\bf{Z}/p$. The first case is simply Corollary~\ref{cor:lefschetz-flat-mu_p}. In the second case, one uses the short exact sequence 
    \[
    0\to \alpha_p \to \bf{G}_a \xr{f\mapsto f^p} \bf{G}_a \to 0
    \]
    to reduce the claim to Definition~\ref{defn:hodge-d-equivalence} with $s=0$. In the last case, one uses the Artin--Schreier sequence to reduce to Definition~\ref{defn:hodge-d-equivalence} again.
\end{proof}

Before we extend Lemma~\ref{lemma:lefschetz-p-power-perfect-field} to more general fields and more general group schemes $G$, we need the following preliminary result:

\begin{lemma}\label{lemma:fppf-descent} Let $X$ be a finite type $k$-scheme, let $X_n$ be the base change $X_{\overline{k}^{\otimes^n_k}}$, and let $G$ be a flat finitely presented commutative group $X$-scheme. Then the natural morphism 
\[
\rm{R}\Gamma_\rm{fppf}(X, G) \to \rm{R}\lim_{n\in \Delta}\left(\rm{R}\Gamma_{\rm{fppf}}(X_{n}, G) \right)
\]
is an isomorphism.
\end{lemma}
\begin{proof}
    For each finite extension $k\subset k'$, denote by $X_{n, k'}$ the fiber product $X_{k'^{\otimes^n_k}}$. The natural map
    \[
    \rm{R}\Gamma_\rm{fppf}(X, G) \to \rm{R}\lim_{n\in \Delta}\left(\rm{R}\Gamma_{\rm{fppf}}(X_{n, k'}, G) \right)
    \]
    is an equivalence for any finite $k\subset k'$ because fppf cohomology satisfies fppf descent. Now \cite[Lemma 2.1]{Cesnavicius2015} implies that the natural morphism
    \[
    \hocolim_{k\subset k'\subset \ov{k}} \rm{R}\Gamma_{\rm{fppf}}(X_{n, k'}, G) \to \rm{R}\Gamma_{\rm{fppf}}(X_{n}, G)
    \]
    is an equivalence for any $n\geq 0$. Thus the claim follows from the fact that totalization of coconnective cosimplicial objects commute with filtered (homotopy) colimits. 
\end{proof}

\begin{cor}\label{cor:lefschetz-p-power} Let $k$ be a field of characteristic $p>0$, let $X \to Y$ be a Hodge $d$-equivalence of syntomic $k$-schemes, and let $G$ be a finite commutative $k$-group scheme of $p$-power order. Then 
\[
C\coloneqq \rm{cone}\left(\rm{R}\Gamma_{\rm{fppf}}\left(Y, G\right) \to \rm{R}\Gamma_{\rm{fppf}}\left(X, G\right)\right) \in D^{\geq d}(\Z).
\]
\end{cor}

\begin{proof}
    Let $X_n = X_{\overline{k}^{\otimes^n_k}}$ (and similarly for $Y$). We have a commutative diagram
    \[
    \begin{tikzcd}
        \rm{R}\Gamma_{\rm{fppf}}(Y, G) \arrow{d} \arrow{r} & \rm{R}\lim_{n\in \Delta}\left(\rm{R}\Gamma_{\rm{fppf}}(Y_{n}, G) \right) \arrow{d} \\
        \rm{R}\Gamma_{\rm{fppf}}(X, G) \arrow{r} & \rm{R}\lim_{n\in \Delta}\left(\rm{R}\Gamma_{\rm{fppf}}(X_{n}, G) \right)
    \end{tikzcd}
    \]
    whose horizontal arrows are isomorphisms by Lemma~\ref{lemma:fppf-descent}. Thus, it suffices to show that 
    \[
    \rm{cone}\Big(\rm{R}\Gamma_{\rm{fppf}}(Y_{n}, G) \to \rm{R}\Gamma_{\rm{fppf}}(X_{n}, G)\Big)
    \]
    lies in $D^{\geq d}(\Z)$. For each finite extension $k\subset k'\subset \ov{k}$, we define $X_{n, k', \ov{k}} = X_{k'^{\otimes^{n-1}_k} \otimes_k \ov{k}}$ (and we define $Y_{n, k', \ov{k}}$ similarly).
    By \cite[Lemma 2.1]{Cesnavicius2015}, it suffices to show that 
    \[
    C_{n, k'} \coloneqq \rm{cone}\Big(\rm{R}\Gamma_{\rm{fppf}}(Y_{n, k', \ov{k}}, G) \to \rm{R}\Gamma_{\rm{fppf}}(X_{n, k', \ov{k}}, G)\Big)
    \]
    lies in $D^{\geq d}(\Z)$. \smallskip
    
    Since syntomic morphisms are closed under pullbacks and compositions, we conclude from Example~\ref{ex:finite-field-extension-syntomic} that each $X_{n, k', \ov{k}}$ and $Y_{n, k', \ov{k}}$ is syntomic over $\ov{k}$ for every finite extension $k\subset k' \subset \ov{k}$. Likewise, Example~\ref{ex:finite-field-extension-syntomic} and \cite[Prop.~5.10(1)]{ABM} imply that $X_{n, k', \ov{k}} \to Y_{n, k', \ov{k}}$ is a Hodge $d$-equivalence for every $k'$. Furthermore, the classification of commutative finite flat $\ov{k}$-group schemes implies that $G_{\ov{k}}$ admits a finite filtration such that each associated graded piece is isomorphic to either $\mu_p$, $\alpha_p$, or $\Z/p\Z$. Therefore, Lemma~\ref{lemma:lefschetz-p-power-perfect-field} implies that $C_{n, k'}$ lives in $D^{\geq d}(\Z)$, as desired.
\end{proof}

\begin{cor}\label{cor:lefschetz-p-power-sufficiently-ample} Let $k$ be a field of characteristic $p>0$, let $X \to Y$ be a Hodge $d$-equivalence of syntomic $k$-schemes, and let $G$ be a finite commutative $k$-group scheme of $p$-power order. Then 
\[
C\coloneqq \rm{cone}\left(\rm{R}\Gamma_{\rm{fppf}}\left(Y_S, G\right) \to \rm{R}\Gamma_{\rm{fppf}}\left(X_S, G\right)\right) \in D^{\geq d}(\Z)
\]
for any syntomic $k$-scheme $S$.
\end{cor}
\begin{proof}
    The closed embedding $X_S \to Y_S$ is a Hodge $d$-equivalence by \cite[Prop.~5.10(i)]{ABM}. Moreover, both $X_S$ and $Y_S$ are syntomic over $k$ because syntomic morphisms are closed under pullbacks and compositions. Therefore, Corollary~\ref{cor:lefschetz-p-power} implies the claim. 
\end{proof}

\begin{thm}\label{thm:main-1} Let $k$ be a field, let $Y$ be a projective syntomic $k$-scheme of pure dimension $N\geq d+1$, let $X\hookrightarrow Y$ be a closed syntomic subscheme, and let $G$ be a finite commutative $k$-group scheme. Then the cone 
\[
\rm{cone}\left(\rm{R}\Gamma_\rm{fppf}\left(Y, G\right) \to \rm{R}\Gamma_\rm{fppf}\left(X, G\right)\right)
\]
lies in $D^{\geq d}(\Z)$ if
\begin{enumerate}
    \item $Y = \bf{P}^N_k$ and $X$ is a global complete intersection of dimension $d$, or
    \item $X\hookrightarrow Y$ is a Cartier divisor such that $Y\smallsetminus X$ is affine and $X \hookrightarrow Y$ is a Hodge $d$-equivalence (e.g.,\, $X\subset Y$ is a strongly ample Cartier divisor). 
\end{enumerate} 
\end{thm}
\begin{proof}
    Let $p\geq 0$ be the characteristic of $k$. We consider the short exact sequence
    \[
    0\to G[p^\infty] \to G \to G/G[p^\infty] \to 0.
    \]
    The group $G' \coloneqq G/G[p^\infty]$ is $p$-torsion-free and finite \'etale. Therefore, it suffices to prove the claim separately for a $p$-power torsion $G[p^\infty]$ and for a $p$-torsion-free \'etale $G'$. \smallskip
    
    The case of a $p$-torsion-free \'etale group scheme follows from Theorem~\ref{thm:etale-lefschetz}. So we can assume that $p > 0$ and $G=G[p^\infty]$ is $p$-power torsion. In either case, $X\to Y$ is a Hodge $d$-equivalence (see Theorem~\ref{thm:derived-lefschetz}), so the claim follows from Corollary~\ref{cor:lefschetz-p-power}. 
\end{proof}

The next example shows that Theorem~\ref{thm:main-1} does not hold for an arbitrary ample divisor $X\subset Y$.

\begin{example}\label{exmpl:thm-fails-for-ample}(\cite[\textsection 2]{Biswas-Holla}, \cite[Ex.\,10.1]{Langer}) Let $k$ be a perfect field of characteristic $p>0$, let $Y$ be a smooth, projective, geometrically connected $k$-scheme, and let $X\subset Y$ be an ample Cartier divisor such that
\begin{enumerate}
    \item $\rm{H}^0(X, \O_X)=k$ (e.g., $X$ is reduced and geometrically connected);
    \item $Y$ is of pure dimension $d+1\geq 2$;
    \item $\rm{H}^1(Y,\O_Y(-X)) \neq 0$. (For examples of such pairs with $d+1=2$, see \cite[Prop.~2.14]{Ekedahl}\footnote{Burt Totaro has informed us that the higher dimensional examples of pairs $(D \subset X)$ constructed in (the proof of) \cite[Th.~2]{Mukai} satisfy the assumption of Example~\ref{exmpl:thm-fails-for-ample}. A full justification of this fact will appear elsewhere.}.) 
\end{enumerate}
Then $r\colon \rm{H}^1_{\rm{fppf}}(Y, \alpha_p) \to \rm{H}^1_{\rm{fppf}}(X, \alpha_p)$ is not injective. In particular, 
\[
C\coloneqq \rm{cone}\big(\rm{R}\Gamma_\rm{fppf}(Y, \alpha_{p}) \rightarrow \rm{R}\Gamma_\rm{fppf}(X, \alpha_{p})\big)
\]
does not lie in $D^{\geq d}(\bf{Z}_p)$.
\end{example}
\begin{proof}
Our assumptions on $Y$ imply that $\rm{H}^0(Y, \O_Y)\simeq k$. Therefore, the map 
\[
\rm{H}^1(Y, \O_Y(-X)) \to \rm{H}^1(Y, \O_Y)
\]
is injective. So any non-trivial class in $\rm{H}^1(Y, \O_Y(-X))$ defines a non-trivial class $x\in \rm{H}^1(Y, \O_Y)$ such that $x|_X = 0 \in \rm{H}^1(X, \O_X)$. We claim that $(F^n_Y)^*(x)=0$ for some $n\geq 0$. Indeed, by functoriality, $(F^n_Y)^*(x)$ lies in $\rm{H}^1(Y, \O_Y(-((F_Y^n)^*X)))=\rm{H}^1(Y, \O_Y(-p^nX))$. Since $d\geq 1$, \cite[\href{https://stacks.math.columbia.edu/tag/0FD8}{Tag 0FD8}]{stacks-project} implies that $\rm{H}^1(Y, \O_Y(-p^nX))=0$ for $n\gg 0$. \smallskip

Choose a minimal $n$ such that $(F_Y^n)^*(x)=0$. Then we replace $x$ with $(F_Y^{n-1})^*(x)$ to assume that $F_Y^*(x)=0$ (and $x\neq 0$). Since $F_Y^*$ and $F_X^*$ are bijective on $\rm{H}^0(Y, \O_Y)$ and $\rm{H}^0(X, \O_X)$ respectively, we may use the Artin--Schreier sequence to conclude that 
\[
\rm{H}^1_{\rm{fppf}}(Y, \alpha_{p}) \simeq \ker(F_Y^* \colon \rm{H}^1(Y, \O_Y) \to \rm{H}^1(Y, \O_Y))
\]
and the same for $X$. In particular, $\rm{H}^1_{\rm{fppf}}(Y, \alpha_{p}) \to \rm{H}^1(Y, \O_X)$ is injective (and the same for $X$). Therefore, $x$ defines a non-trivial class in $\rm{ker}(r)$. In particular, $C$ does not lie in $D^{\geq d}(\Z_p)$.
\end{proof}

If $G$ is a lisse sheaf of $\bf{F}_\ell$-modules on $Y_\et$ with $\ell\neq \rm{char} \, k$, then the Lefschetz hyperplane theorem holds for $G$ if $Y$ is smooth. One may wonder if there is an analogous result for flat coefficients. In general, we do not know the correct coefficient theory for flat cohomology in which to pose such a question. In any event, Theorem~\ref{thm:main-1} is false if one does not assume that $G$ comes from a base field, as the following example shows.

\begin{example}\label{example:fails-for-general-finite-group} Let $p$ be a prime number, and let $k=\bf{F}_p$ be a finite field with $p$ elements. Then, for any $N>1$, there is a commutative finite flat rank $p$ group scheme $G$ on $\bf{P}^N_k$ such that 
\begin{enumerate}
    \item Zariski-locally on $\bf{P}^N_k$, the group $G$ is defined over $k$,
    \item for any hyperplane $H\subset \bf{P}^N_k$, the cone 
    \[
        C\coloneqq \rm{cone}\left(\rm{R}\Gamma_\rm{fppf}\left(\bf{P}^N_k, G\right) \to \rm{R}\Gamma_{\rm{fppf}}\left(H, G\right)\right)
    \]
    does not lie in $D^{\geq N-1}(\Z)$.
\end{enumerate}
\end{example}
\begin{proof}
    Let $\bf{G}_a(n)$ be the $\bf{P}^N$-group scheme associated with the line bundle $\O(n)$. Then we define $G=\ker(\rm{Fr}\colon \bf{G}_a(1)\to \bf{G}_a(p))$. Note, Zariski locally on $\bf{P}^N$, the group $G$ is isomorphic to $\alpha_p$ and in particular defined over $k$. \smallskip
    
    Now using that $\rm{H}^i_\rm{fppf}(\bf{P}^N_k, \bf{G}_a(n))=\rm{H}^i(\bf{P}^N_k, \O(n))$, Serre's calculation of cohomology groups of $\O(n)$, and the short exact sequence $0 \to G \to \bf{G}_a(1) \to \bf{G}_a(p) \to 0$, we conclude that 
    \[
    \log_p \left(\#\rm{H}^1_\rm{fppf}(\bf{P}^{N}_k, G)\right)= \binom{N+p}{N} - N -1
    \]
    \[
    \log_p \left(\#\rm{H}^1_\rm{fppf}(H, G)\right) = \log_p \left(\# \rm{H}^1_\rm{fppf}(\bf{P}^{N-1}_k, G)\right) = \binom{N+p-1}{N-1} - N.
    \]
    In particular, the map $\rm{H}^1_\rm{fppf}(\bf{P}^{N}_k, G) \to \rm{H}^1_\rm{fppf}(H, G)$ can not be injective by cardinality reasons. Therefore, $C$ cannot lie in $D^{\geq N-1}(\Z_p)$ for any $N>1$.   
\end{proof}

\subsection{The torsion part of the Picard scheme}

Fix a field $k$. In this section, we use the results of Section~\ref{section:lefschetz-p-part} and Section~\ref{section:lefschetz-non-p-part} to get a Lefschetz hyperplane theorem for the torsion part of the Picard group. We show that, for a complete intersection $X\subset \bf{P}^N_k$ scheme of dimension at least $2$, the torsion part of the Picard group $\rm{Pic}(X)_{\rm{tors}}$ and the torsion component $\bf{Pic}^\tau_{X/k}$ vanish. We also give a version of this result for a general strongly ample divisor. \smallskip

If $\dim X \geq 3$, Grothendieck proved the stronger result $\rm{Pic}(X) \simeq \bf{Z}$ in \cite[Exp.\ XII, Cor.~3.2]{SGA2} which can be also used to deduce that $\bf{Pic}_{X/k} \simeq \ud{\Z}$ as $k$-group schemes. These results are sharp: for instance, the Segre embedding realizes $\bf{P}^1_k\times \bf{P}^1_k$ as a hypersurface in $\bf{P}^3_k$, but $\rm{Pic}(\bf{P}^1_k\times \bf{P}^1_k) \simeq \bf{Z} \oplus \bf{Z}$. However, one can still control the torsion part of Picard group in dimension $2$. If $X$ is a smooth surface and $k$ is algebraically closed, then these results were established in \cite[Exp.\ XI, Th.~1.8]{SGA7_2}. The general case was proven in \cite[Cor.~7.2.3]{Ces-Scholze} by different methods.

\begin{proof}[Proof of Theorem~\ref{thm:intro-main-3}]
For (\ref{thm:SGA7-1}), it suffices to show that $\rm{Pic}(X)[p]=0$ for {\it every} prime number $p$. Since $\rm{Pic}(\bf{P}^N_k)[p]=0$, it suffices to show that the natural morphism $\rm{Pic}(\bf{P}^N_k)[p] \to \rm{Pic}(X)[p]$ is an isomorphism. From the Kummer sequence we obtain the following commutative diagram: 
\begin{equation*}\label{equation:picard-diagram-first}
\begin{tikzcd}
\O(\bf{P}^N_k)^\times = \rm{H}^0_{\rm{fppf}}(\bf{P}^N_k, \bG_m) \arrow[r] \arrow[d, "\alpha"]
    &\rm{H}^1_{\rm{fppf}}(\bf{P}^N_k, \mu_{p}) \arrow[r] \arrow[d, "\beta"]
    &\rm{Pic}(\bf{P}^N_k)[p] \arrow[r] \arrow[d, "\gamma"]
    &0 \\
\O(X)^\times = \rm{H}^0_{\rm{fppf}}(X, \bG_m) \arrow[r]
    &\rm{H}^1_{\rm{fppf}}(X, \mu_{p}) \arrow[r]
    &\rm{Pic}(X)[p] \arrow[r]
    &0.
\end{tikzcd}
\end{equation*}
    We note that the morphism $\O(\bf{P}^N_k) \to \O(X)$ is an isomorphism due to Theorem~\ref{thm:derived-lefschetz} (and Definition~\ref{defn:hodge-d-equivalence} with $s=0$). In particular, $\alpha$ is an isomorphism as well. Also, $\beta$ is an isomorphism by Theorem~\ref{thm:derived-lefschetz}  and Theorem~\ref{thm:main-1}. By the five lemma, $\gamma$ is an isomorphism, and (\ref{thm:SGA7-1}) holds. \smallskip

    Now we show (\ref{thm:SGA7-2}). It suffices to show that the class of $[\O_X(1)] \in \rm{Pic}(X)$ has non-zero image in $\rm{Pic}(X)/p$ for {\it every} prime $p$. Let $c^X_1\colon \rm{Pic}(X)/p \to \rm{H}^2_{\rm{fppf}}(X,\mu_p)$ be the connecting map from the Kummer exact sequence. By definition, $c^X_1$ is injective, so it is enough to show that $c^X_1([\O_X(1)]) \neq 0$ in $\rm{H}^2_{\rm{fppf}}(X, \mu_p)$. The commutative square
    \[
    \begin{tikzcd}
    \rm{Pic}(\bf{P}^N_k)/p \arrow{d}{\rm{res}^\rm{Pic}} \arrow{r}{c_{1}^{\bf{P}^N}} & \rm{H}^2_{\rm{fppf}}(\bf{P}^N_k, \mu_p) \arrow{d}{\rm{res}^{\rm{fppf}}} \\
    \rm{Pic}(X)/p \arrow{r}{c_{1}^X} & \rm{H}^2_{\rm{fppf}}(X, \mu_p).
    \end{tikzcd}
    \]
    shows that we have $c_{1}^X([\O_X(1)])=\rm{res}^{\rm{fppf}}\big(c_{1}^{\bf{P}^n}([\O_{\bf{P}^N}(1)])\big)$. Then we conclude that $c_1^X([\O_X(1)])\neq 0$ because $c_1^{\bf{P}^N}([\O_{\bf{P}^N}(1)])$ is a generator of $\rm{Pic}(\bf{P}^N_k)\simeq \bf{Z}$, and $\rm{res}^{\rm{fppf}}$ is injective by Theorem~\ref{thm:main-1}. \smallskip

    Finally, we show (\ref{thm:SGA7-3}). For this, we can and do assume that $k=\ov{k}$. By \cite[Exp.\ XII, Cor.\ 1.5, Exp.\ XIII, Thm.\ 4.7]{SGA6}, $\bf{Pic}^\tau_{X/k}$ is a finite type $k$-group scheme and an open subfunctor of $\bf{Pic}_{X/k}$, so $\rm{T}_e(\bf{Pic}^\tau_{X/k})=\rm{T}_e(\bf{Pic}_{X/k})\simeq \rm{H}^1(X, \O_X)$. Theorem~\ref{thm:derived-lefschetz} implies that $\rm{H}^1(X, \O_X)=0$ and so $\bf{Pic}^\tau_{X/k}$ is \'etale. Since $k = \ov{k}$, we have $\bf{Pic}^\tau_{X/k}(\overline{k}) = \rm{Pic}(X)_{\rm{tors}} = 0$ by the above, and (\ref{thm:SGA7-3}) follows. 
\end{proof}

Our next aim is to prove an analogue of Theorem~\ref{thm:intro-main-5} for strongly ample divisors (see Theorem~\ref{thm:SGA7-suff-ample}). The derivation of this is slightly more involved than in the previous case, and we must begin with a series of results about algebraic groups.

\begin{lemma}\label{lemma:torsion-in-quotient}
Let $G$ be a commutative group scheme locally of finite type over a field $k$, and let $H$ be a finite type closed $k$-subgroup scheme of $G$. Let $n$ be a positive integer. For each $M \geq 1$, there exists a closed subgroup scheme $G_M$ of $G$ killed by some power of $n$ such that the map
\[
G_M \to G/H
\]
factors through $(G/H)[n^M]$ and such that the factored map is an epimorphism of fppf sheaves.\footnote{A homomorphism of finite type group schemes over a field induces an epimorphism of fppf sheaves if and only if it is faithfully flat by \cite[Exp.\,VI$_\text{B}$, Rmk.\ 9.2.2]{SGA3}, but we do not need this fact.} (Note that $G/H$ exists as a group scheme by \cite[Exp.\ VI\textsubscript{A}, Thm.\ 3.2]{SGA3}.)
\end{lemma}

\begin{proof}
Fixing the positive integer $M$, we may replace $G$ by the schematic preimage of $(G/H)[n^M]$ in $G$ to assume that $G/H$ is $n^M$-torsion. Consider for any $N$ the commutative diagram
\[
\begin{tikzcd}
0 \arrow[r] 
    &H \arrow[r] \arrow[d, "n^N"]
    &G \arrow[r] \arrow[d, "n^N"]
    &G/H \arrow[r] \arrow[d, "n^N"]
    &0 \\
0 \arrow[r]
    &H \arrow[r]
    &G \arrow[r]
    &G/H \arrow[r]
    &0
\end{tikzcd}
\]
with exact rows. By the snake lemma, this gives an exact sequence of group schemes
\[
G[n^N] \to (G/H)[n^N] \to H/n^N H \to G/n^N G.
\]
So it suffices to show that there is some $N \geq M$ such that $H/n^N H \to G/n^N G$ is a monomorphism. Equivalently, one must show that there is some $N \geq M$ such that
\[
H \cap n^N G = n^N H.
\]

To prove the existence of some such $N$, note that the sequence $\{n^N H\}$ of closed $k$-subgroup schemes of $H$ is decreasing, so the fact that $H$ is noetherian (being finite type over $k$) implies that there exists some $N_0$ such that $n^{N_0} H = n^{N_0+1} H$. Since $G/H$ is $n^M$-torsion, we have
\[
H \cap n^{N_0 + M} G \subset n^{N_0} H = n^{N_0 + M} H
\]
by the choice of $N_0$. Thus, taking $N \coloneqq  N_0 + M$ completes the proof.
\end{proof}

\begin{lemma}\label{lemma:torsion-is-dense}
Let $G$ be a commutative group scheme of finite type over an algebraically closed field $k$ of characteristic $p \geq 0$.
\begin{enumerate}
    \item\label{lemma:torsion-is-dense-1} The natural map $G(k)_{\rm{tors}} \to G(k)/G^0(k)$ is surjective.
    \item\label{lemma:torsion-is-dense-2} For some integer $N \geq 1$, the map $G[p^N] \to G/G_{\rm{red}}$ is an epimorphism of fppf sheaves.
    \item\label{lemma:torsion-is-dense-3} If $G$ is smooth and $p > 0$, then $G(k)_{\rm{tors}}$ is schematically dense in $G$.
\end{enumerate}
\end{lemma}

\begin{proof}
For (1), it suffices to show that for every integer $n \geq 1$, the natural map $G(k)[n^\infty] \to (G(k)/G^0(k))[n^\infty]$ is surjective, and this follows from Lemma~\ref{lemma:torsion-in-quotient}. \smallskip

For (2), there is nothing to prove if $p = 0$. If instead $p > 0$, then $G/G_{\rm{red}}$ is a finite $k$-group scheme whose order is a power of $p$. By a theorem of Deligne \cite[Sec.\ 1]{Oort-Tate}, the group scheme $G/G_{\rm{red}}$ is killed by its order, so the result follows again from Lemma~\ref{lemma:torsion-in-quotient}. \smallskip

Finally, we consider (3). By (1), we may and do reduce to the case that $G$ is connected, in which case we will show that if $\ell \neq p$ is any prime number then $G(k)[(\ell p)^\infty]$ is schematically dense in $G$. By a theorem of Chevalley \cite[Th.~1.1]{Conrad-Chevalley}, since $k$ is perfect there is a short exact sequence
\[
0 \to H \to G \to A \to 0,
\]
where $H$ is a linear algebraic group over $k$ and $A$ is an abelian variety over $k$. Moreover, since $H$ is a commutative linear algebraic group over a perfect field, we have $H = T \times U$ for some $k$-torus $T$ and a smooth commutative unipotent $k$-group scheme $U$. It is standard that $T(k)[\ell^\infty]$ and $A(k)[\ell^\infty]$ are schematically dense in $T$ and $A$, respectively, and $U = U[p^M]$ for some $M \geq 1$. (This is where we use that $p > 0$; in characteristic $0$, the group $\bG_a$ has no torsion.) \smallskip

Let $G_0$ denote the schematic closure of $G(k)[(\ell p)^\infty]$ in $G$, so that $G_0$ is a smooth closed $k$-subgroup scheme of $G$. We aim to show that $G_0 = G$. Every connected commutative finite type $k$-group scheme is $\ell$-divisible, so by the snake lemma the natural map $G(k)[\ell^\infty] \to A(k)[\ell^\infty]$ is surjective. By schematic density of $A(k)[\ell^\infty]$ in $A$, the induced map $G_0 \to A$ is dominant, hence surjective by \cite[Exp.\ VI\textsubscript{B}, Prop.\ 1.2]{SGA3}.
It suffices therefore to show that $H$ is contained in $G_0$. But $H(k)[(\ell p)^\infty]$ is schematically dense in $H$, so indeed $H \subset G_0$ and so $G_0 = G$, establishing the result.
\end{proof}

\begin{lemma}\label{lemma:isomorphism-criterion}
Let $f\colon G \to H$ be a homomorphism of commutative group schemes of finite type over a field $k$ of characteristic $p \geq 0$. Suppose that
\begin{enumerate}
    \item $f[\ell^n](\ov{k})\colon G[\ell^n](\ov{k}) \to H[\ell^n](\ov{k})$ is an isomorphism for every prime $\ell$ and every $n \geq 1$,
    \item $\Lie f\colon \Lie G \to \Lie H$ is an isomorphism,
    \item if $p > 0$, then $f[p^n]\colon G[p^n] \to H[p^n]$ is an epimorphism of fppf sheaves for every $n \geq 1$.
\end{enumerate}
Then $f$ is an isomorphism.
\end{lemma}

\begin{proof}
We may and do assume that $k$ is algebraically closed. (2) shows that $\ker f$ is finite etale, and so (1) implies that $\ker f = 0$. Thus, $f$ is a closed embedding by \cite[Exp.\ VI\textsubscript{B}, Cor.\ 1.4.2]{SGA3}. Moreover, Lemma~\ref{lemma:torsion-is-dense}(\ref{lemma:torsion-is-dense-1}) shows that the image of $f$ intersects each connected component of $H$ nontrivially. If $p = 0$ then $G$ and $H$ are smooth, so $f$ is an isomorphism. Thus from now on we assume $p > 0$. \smallskip

Let $\overline{G} = G/G_{\rm{red}}$ (which is a scheme by \cite[Exp.\ VI\textsubscript{A}, Thm.\ 3.2]{SGA3}) and consider the commutative diagram
\[
\begin{tikzcd}
0 \arrow[r]
    &G_{\rm{red}} \arrow[r] \arrow[d, "f_{\rm{red}}"]
    &G \arrow[r] \arrow[d, "f"]
    &\overline{G} \arrow[r] \arrow[d, "\overline{f}"]
    &0 \\
0 \arrow[r]
    &H_{\rm{red}} \arrow[r]
    &H \arrow[r]
    &\overline{H} \arrow[r]
    &0
\end{tikzcd}
\]
with exact rows. Since $p> 0$ and $f$ is surjective on torsion,  Lemma~\ref{lemma:torsion-is-dense}(\ref{lemma:torsion-is-dense-3}) shows that $f$ is dominant, and thus it is surjective by \cite[Exp.\ VI\textsubscript{B}, Prop.\ 1.2]{SGA3}. Now $G_{\rm{red}}$ and $H_{\rm{red}}$ are both smooth over $k$, so because $f$ is a surjective closed embedding it follows that $f_{\rm{red}}$ is an isomorphism. Thus to show that $f$ is an isomorphism, it suffices to show that $\overline{f}$ is an isomorphism. \smallskip

Now by Lemma~\ref{lemma:torsion-is-dense}(\ref{lemma:torsion-is-dense-2}), there is some integer $N \geq 1$ such that the natural maps $G[p^N] \to \overline{G}$ and $H[p^N] \to \overline{H}$ are faithfully flat. Thus we find a commutative diagram
\[
\begin{tikzcd}
0 \arrow[r]
    &G_{\rm{red}}[p^N] \arrow[r] \arrow[d, "f_{\rm{red}}{[p^N]}"]
    &G[p^N] \arrow[r] \arrow[d, "f{[p^N]}"]
    &\overline{G} \arrow[r] \arrow[d, "\overline{f}"]
    &0 \\
0 \arrow[r]
    &H_{\rm{red}}[p^N] \arrow[r]
    &H[p^N] \arrow[r]
    &\overline{H} \arrow[r]
    &0
\end{tikzcd}
\]
with exact rows. By (3) and the fact that $\ker f = 0$, the map $f[p^N]$ is an isomorphism. The previous paragraph shows that $f_{\rm{red}}[p^N]$ is an isomorphism, so also $\overline{f}$ is an isomorphism. Since both $f_{\rm{red}}$ and $\overline{f}$ are isomorphisms, we see that $f$ is an isomorphism, as desired.
\end{proof}

With these preliminaries in hand, we can finally prove the following theorem.

\begin{thm}\label{thm:SGA7-suff-ample}
Let $Y$ be a projective syntomic $k$-scheme of pure dimension $\geq 3$, and let $X \subset Y$ be a Cartier divisor such that $Y\smallsetminus X$ is affine and $X \hookrightarrow Y$ is a Hodge $2$-equivalence. The natural map $f\colon\Pic^\tau_{Y/k} \to \Pic^\tau_{X/k}$ is an isomorphism.
\end{thm}

\begin{proof}
We may and do assume that $k$ is algebraically closed of characteristic $p \geq 0$ by \cite[Prop.\ 5.10]{ABM}. We need only verify the hypotheses of Lemma~\ref{lemma:isomorphism-criterion} applied to $G = \Pic^\tau_{Y/k}$ and $H = \Pic^\tau_{X/k}$ (which are finite type $k$-group schemes by \cite[Exp.\ XII, Cor.\ 1.5;  Exp.\ XIII, Thm.\ 4.7(iii)]{SGA6}). Condition (2) follows from the fact that the natural map $\rm{H}^1(Y, \O) \to \rm{H}^1(X, \O)$ is an isomorphism by Definition~\ref{defn:sufficiently-ample}.\smallskip

Let $\ell$ is a prime number and let $S$ be a syntomic $k$-scheme. We have a commutative diagram 
\begin{equation}\label{equation:picard-diagram}
\begin{tikzcd}
\rm{H}^0(Y_S, \bG_m) \arrow[r] \arrow[d]
    &\rm{H}^1(Y_S, \mu_{\ell^n}) \arrow[r] \arrow[d]
    &\rm{Pic}(Y_S)[\ell^n] \arrow[r] \arrow[d]
    &0 \\
\rm{H}^0(X_S, \bG_m) \arrow[r]
    &\rm{H}^1(X_S, \mu_{\ell^n}) \arrow[r]
    &\rm{Pic}(X_S)[\ell^n] \arrow[r]
    &0
\end{tikzcd}
\end{equation}
with exact rows. The leftmost vertical arrow of (\ref{equation:picard-diagram}) is an isomorphism by Definition~\ref{defn:hodge-d-equivalence} and \cite[Prop.\ 5.10]{ABM}, since $\rm{H}^0(Y_S, \bG_m) = \rm{H}^0(Y_S, \O)^*$. Moreover, the map
\[
\rm{H}^1(Y_S, \mu_{\ell^n}) \to \rm{H}^1(X_S, \mu_{\ell^n})
\]
is an isomorphism by Theorem~\ref{thm:main-1}. Thus $\rm{Pic}(Y_S)[\ell^n] \to \rm{Pic}(X_S)[\ell^n]$ is an isomorphism by the five lemma. Since $X$ has a rational point ($k$ being algebraically closed), we have $\Pic_{Y/k}(S) = \rm{Pic}(Y_S)/\rm{Pic}(S)$, and similarly for $X$. Letting $S = \Spec k$, we see that $\Pic^\tau_{Y/k}[\ell^n](k) \to \Pic^\tau_{X/k}[\ell^n](k)$ is an isomorphism and we have verified (1).\smallskip

Now set $S = P_{X,n} \coloneqq \Pic_{X/k}[p^n]$, which is a syntomic $k$-scheme by Lemma~\ref{lemma:finite-flat-groups-syntomic}. Define $P_{Y,n}$ similarly. Since $\rm{Pic}(P_{X,n}) = 0$ (as $P_{X, n}$ is an extension of a finite $k$-group scheme by a unipotent group scheme), the argument above shows $P_{X,n}(P_{X,n}) = \rm{Pic}(X_{P_{X,n}})[p^n]$, and similarly for $P_{Y,n}(P_{X,n})$. Thus the previous paragraph shows that the natural map $P_{Y,n}(P_{X,n}) \to P_{X,n}(P_{X,n})$ is an isomorphism, and so there is a morphism $g\colon P_{X,n} \to P_{Y,n}$ such that $f \circ g = \rm{id}_{P_{X,n}}$. Therefore $f$ is an epimorphism of fppf sheaves, so Lemma~\ref{lemma:isomorphism-criterion} shows that $\Pic^\tau_{Y/k} \to \Pic^\tau_{X/k}$ is an isomorphism.
\end{proof}

\begin{cor}\label{cor:Pic-scheme-iso}
Let $Y$ be a projective syntomic $k$-scheme of pure dimension $\geq 4$, and let $X \subset Y$ be a Cartier divisor such that $Y\smallsetminus X$ is affine, $X \hookrightarrow Y$ is a Hodge $2$-equivalence, and $\rm{H}^i(Y, \O_Y(-nX)) = 0$ for all $n > 0$ and $i = 1, 2$. The map $\Pic_{Y/k} \to \Pic_{X/k}$ is an isomorphism.
\end{cor}

\begin{proof}
We may and do assume $k = \ov{k}$. In view of Theorem~\ref{thm:SGA7-suff-ample}, we see that $\Pic^\tau_{Y/k} \to \Pic^\tau_{X/k}$ is an isomorphism. Considering the commutative diagram
\[
\begin{tikzcd}
0 \arrow[r]
    &\Pic^\tau_{Y/k} \arrow[r] \arrow[d]
    &\Pic_{Y/k} \arrow[r] \arrow[d]
    &\rm{Num}(Y) \arrow[r] \arrow[d]
    &0 \\
0 \arrow[r]
    &\Pic^\tau_{X/k} \arrow[r]
    &\Pic_{X/k} \arrow[r]
    &\rm{Num}(X) \arrow[r]
    &0
\end{tikzcd}
\]
we see that it suffices to show that the natural map $\rm{Num}(Y) \to \rm{Num}(X)$ is an isomorphism, and for this it suffices to show that $\rm{Pic}(Y) \to \rm{Pic}(X)$ is an isomorphism, a consequence of \cite[Exp.\ XII, Cor.\ 3.6]{SGA2} (which applies because $Y$ is syntomic and $\rm{H}^i(Y, \O_Y(-nX))=0$ for $n>0$ and $i=1,2$).
\end{proof}

We note that Corollary~\ref{cor:example-nef-normal-bundle} and Theorem~\ref{thm:derived-lefschetz} (together with \cite[Cor.~2.8]{Deligne-Illusie}) produce many interesting examples when the assumptions of the two above theorems are satisfied. \smallskip

The following example provides some evidence that Theorem~\ref{thm:SGA7-suff-ample} and Corollary~\ref{cor:Pic-scheme-iso} are probably false for a general ample divisor:

\begin{example}\label{rmk:no-picard-lefschetz-ample} Let $X\subset Y$ be a pair as in Example~\ref{exmpl:thm-fails-for-ample} with $d+1\geq 3$ (resp. $d+1\geq 4$). Then 
\begin{enumerate}
    \item\label{rmk:no-picard-lefschetz-ample-1} the map $\bf{Pic}^\tau_{Y/k} \to \bf{Pic}^\tau_{X/k}$ (resp. $\bf{Pic}_{Y/k} \to \bf{Pic}_{X/k}$) is not an isomorphism;
    \item\label{rmk:no-picard-lefschetz-ample-2} if $X$ is also smooth and the natural morphism $\rm{H}^0(Y,\Omega^1_Y)\to \rm{H}^0(X, \Omega^1_X)$ is an isomorphism, then $\rm{Pic}(Y)[p] \to \rm{Pic}(X)[p]$ is not an isomorphism.\footnote{We do not know if such examples exist.}
\end{enumerate}
\end{example}
\begin{proof}
    First, we note that the proof of Example~\ref{exmpl:thm-fails-for-ample} implies that the natural morphism $\rm{H}^1(Y, \O_Y) \to \rm{H}^1(X, \O_X)$ is not an isomorphism. Thus, the morphism $\bf{Pic}^\tau_{Y/k} \to \bf{Pic}^\tau_{X/k}$ (resp. $\bf{Pic}_{Y/k} \to \bf{Pic}_{X/k}$) does not induce an isomorphism of Lie algebras. \smallskip

    For (\ref{rmk:no-picard-lefschetz-ample-2}), we note that the Kummer exact sequence implies that $\rm{H}^1(Y, \mu_p) = \rm{Pic}(Y)[p]$ and the same for $X$. Therefore, it suffices to show the analogous claim for $\rm{H}^1(-, \mu_p)$. Let $Y'$ be the Frobenius twist of $Y$, and let $\Omega^1_{Y'/k, \rm{cl}}$ denote the sheaf of closed $1$-forms on $Y'$ (and similarly for $X$ and $X'$). Now we note that \cite[Prop.~2.4]{Artin-Milne} implies that there are exact sequences
    \begin{equation}\label{eqn:mu_p}
    0 \to \rm{H}^1(Y, \mu_p) \to \rm{H}^0(Y, \Omega^1_{Y'/k, \rm{cl}}) \xrightarrow{\psi_1} \rm{H}^0(Y, \Omega^1_{Y/k}),
    \end{equation}
    \begin{equation}\label{eqn:alpha_p}
    0 \to \rm{H}^1(Y, \alpha_p) \to \rm{H}^0(Y, \Omega^1_{Y'/k, \rm{cl}}) \xrightarrow{\psi_2} \rm{H}^0(Y, \Omega^1_{Y/k}),
    \end{equation}
    and similarly for $X$ (we note that $\psi_1$ and $\psi_2$ are {\it different} maps). Now (the proof of) Example~\ref{exmpl:thm-fails-for-ample} ensures that the natural morphism $\rm{H}^1(Y, \alpha_p) \to \rm{H}^1(X, \alpha_p)$ is not injective. Thus (\ref{eqn:alpha_p}) formally implies that the natural morphism $\rm{H}^0(Y, \Omega^1_{Y'/k, \rm{cl}}) \to \rm{H}^0(X, \Omega^1_{X'/k, \rm{cl}})$ is not injective. Now (\ref{eqn:mu_p}) and our assumption imply that $\rm{H}^1(Y, \mu_p) \to \rm{H}^1(X, \mu_p)$ is not injective either. This finishes the proof. 
\end{proof}

\bibliography{biblio}

\end{document}